\documentclass[a4paper,12pt]{article}
\usepackage{amsmath,amssymb,amsfonts,amscd,amsthm}
\usepackage{graphicx}
\usepackage{dsfont}
\usepackage{times}
\usepackage{a4wide}
\usepackage{blindtext}
\usepackage[titletoc, toc, title, page]{appendix}
\usepackage{color}
\usepackage{enumerate}
\usepackage{fancyhdr}		
\usepackage{pgf,tikz}
\usepackage{sectsty} \definecolor{bleu1}{RGB}{0,57,128}
\usepackage[utf8]{inputenc}
\usepackage{amsfonts}
\usepackage{bbm}
\usepackage[english]{babel}
\usepackage[pagewise]{lineno}
\allowdisplaybreaks

\usetikzlibrary{arrows}

\newtheorem{theorem}{Theorem}[section]                             
\newtheorem{lemma}[theorem]{Lemma}
\newtheorem{proposition}[theorem]{Proposition}

\newtheorem{remark}[theorem]{Remark}
\newtheorem{definition}[theorem]{Definition}

 \newtheorem{main}{Theorem}

\renewenvironment{proof}{\noindent {\bf Proof.}}{\qed\vskip 0.2cm}

\newcommand{\C}{\mathbb C}
\newcommand{\R}{\mathbb R}
\newcommand{\Q}{\mathbb Q}
\newcommand{\Z}{\mathbb Z}
\newcommand{\N}{\mathbb N}

\newcommand{\T}{\mathbb{T}}

\makeatletter
\newenvironment{proofof}[1]{\par
  \pushQED{\qed}%
  \normalfont \topsep6\p@\@plus6\p@\relax
  \trivlist
  \item[\hskip\labelsep
        \bfseries
    \textit{ Proof of #1\@addpunct{.}}]\ignorespaces
}{%
  \popQED\endtrivlist\@endpefalse
}
\makeatother

\renewcommand{\inf}{\mathop{\mbox{{inf}}}
\renewcommand{\sup}{\mathop{\mbox{sup}}}
\renewcommand{\lim}{\mathop{\mbox{lim}}}
\renewcommand{\max}{\mathop{\mbox{max}}}}
\renewcommand{\min}{\mathop{\mbox{min}}}

\providecommand{\abs}[1]{\lvert#1\rvert}		
\providecommand{\norm}[1]{\lVert#1\rVert}

\makeatletter
\def\author#1{\gdef\autrun{\def\and{\unskip, }#1}\gdef\@author{#1}}
\def\address#1{{\def\and{\\\hspace*{18pt}}\renewcommand{\thefootnote}{}%
\footnote {#1}}}
\makeatother
\def\email#1{e-mail: #1}

\def\keywords#1{\par\medskip
\noindent\textbf{Keywords.} #1}

\usepackage{caption}
\usepackage{pgf,tikz}
\usetikzlibrary{arrows}

\usepackage[colorlinks,linkcolor=bleu1,anchorcolor=green,citecolor=bleu1]{hyperref}
\usepackage{xcolor}

\definecolor{bluegray}{rgb}{0.4, 0.6, 0.8}

\usepackage{xcolor}

\numberwithin{equation}{section}
\begin{document}

\definecolor{ttzzqq}{rgb}{0.2,0.6,0}
\definecolor{uququq}{rgb}{0.25,0.25,0.25}
\definecolor{ccqqqq}{rgb}{0.8,0,0}
\definecolor{ttttff}{rgb}{0.2,0.2,1}

\baselineskip=17pt
\title{Analytic weakly mixing diffeomorphisms on 
odd dimensional spheres}
\date{}
\author{Gerard Farr\'e}
\maketitle
\address{G.Farré: KTH, Department of Mathematics; \email{gerardfp@kth.se}}

\begin{abstract}
 We present an approximation by conjugation scheme to obtain analytic diffeomorphisms of odd dimensional spheres that are weakly mixing with respect to the volume. 
\end{abstract}

\section{Introduction}

\label{sec_11}
This paper is devoted to the study of the phase space dependence of topological and measurable properties of dynamical systems, which aims at 
 identifying the class of manifolds that allow a map of a certain regularity to exhibit
 a given topological or ergodic property. In particular we present a constructive method to prove the existence of weakly mixing and analytic diffeomorphisms with respect to the volume on odd dimensional spheres.  
The type of methods that we use appeared first in \cite{AK}, where the authors proved 
that any compact, connected and smooth manifold supporting a non trivial free $\mathbb{S}^1$-action admits smooth and weakly mixing diffeomorphisms.
These constructions, which are usually referred to as \textit{Anosov-Katok} or approximation by conjugation (AbC) schemes, are performed using an iterative procedure starting with a simple function coming from the free $\mathbb{S}^1$-action, and the maps obtained at the end can be chosen to lie arbitrarily close to this initial map in the smooth topology.
For particular manifolds these results have been extended to the analytic category. In the case of unique ergodicity
specific constructions were carried out in \cite{Sap} for the two torus and in \cite{FK} for odd dimensional spheres. We also refer the reader to \cite{Kunde}, \cite{BANERJEE_2018} for other related results involving constructions of analytic diffeomorphisms on $d$-dimensional tori ($d\geq 2$) and to \cite{FK} for a discussion on the difficulties of extending the method for analytic maps.
Whether or not AbC methods can be used to obtain analytic weakly mixing diffeomorphisms on odd dimensional spheres becomes then a natural question. Before entering into details, let us recall how the \textit{Anosov-Katok} method works. For a complete and modern review on the method and some of its applications, we refer again the reader to \cite{faka}.

Given a connected, compact and smooth manifold $\mathcal{M}$ (equipped with a smooth measure $\mu$) admitting a non trivial smooth free $\mathbb{S}^1$-action $\{S_t\}_{t\in \T}$,
one constructs a sequence of $\mu$-preserving diffeomorphisms $H_n:\mathcal{M}\to \mathcal{M}$ and ${\{\alpha_n\}\subset \Q}$ in such a way that
\begin{equation}
\label{eq_intro}
 F=\lim_{n\to \infty} F_n, \quad  F_n= H_n^{-1} \circ S_{\alpha_{n+1}} \circ H_n
\end{equation} 
 satisfies the ergodic property of interest. It is due to an appropriate inductive choice of ``wild'' conjugacies $H_n$ and a fast converging sequence $\{\alpha_n\}$  that one can induce the desired ergodic behaviour for the limit map while ensuring the convergence of $\{F_n\}$. 
On the other hand the sequence $\{S_{\alpha_n}\}$ also tends to  $S_{\tilde{\alpha}}$ for some $\tilde{\alpha}\in \T$. Thus $\{H_n\}$ cannot converge, since otherwise $F$ would be conjugate to $S_{\tilde{\alpha}}$ and the construction would become trivial. Nevertheless $F$ lies in the closure of the space of maps conjugate to $S_{\tilde{\alpha}}$,
\begin{equation*}
 \mathcal{A}_{\tilde{\alpha}}(\mathcal{M})= \{H^{-1} \circ S_{\tilde{\alpha}} \circ H \; | \; H \in \text{Diff}_{\mu}(\mathcal{M})\},
\end{equation*}
where we have denoted by $\text{Diff}_{\mu}(\mathcal{M})$ the space of diffeomorphisms of $\mathcal{M}$ preserving $\mu$.
Thus the \textit{Anosov-Katok} method can be used to obtain maps with ergodic properties lying on the closure of a space of integrable maps.  An illustrative example of this fact can be found in \cite{FS}, where smooth area preserving weakly mixing maps of the disc $F\in \overline{\mathcal{A}_{\tilde{\alpha}}(\mathbb{D})}$ were built with $\tilde{\alpha}$ an arbitrary Liouville number. Recall that given a measure space $(\mathcal{M}, \mathcal{B}, \mu)$, a $\mu$-preserving map $F:\mathcal{M}\to \mathcal{M} $ is said to be weakly mixing if for all $A, B \in \mathcal{B}$ we have 
\begin{equation*}
 \lim_{n\to \infty}{} \frac{1}{n}\sum_{j=0}^{n-1}{\abs{\mu(F^{-j}(A)\cap B)- \mu(A)\mu(B)}}=0,
\end{equation*}
or equivalently if $F$ possesses no nonconstant eigenfunctions. We will use instead, as it was done in \cite{FS}, the following characterization of weak mixing (see \cite{Sklover}).
\begin{definition}[Weak mixing]
 A $\mu$-preserving diffeomorphism $F:\mathcal{M} \rightarrow \mathcal{M}$ is weakly mixing if there exists an increasing sequence $\{m_n\}_n\subset \N$ such that for any pair of measurable sets $A,B \in \mathcal{B}$ 
 \begin{equation}
 \label{eq_weak_mix}
  \lim_{n\to \infty} \mu(A\cap F^{-m_n}(B)) = \mu(A)\mu(B).
 \end{equation}
\end{definition}
In our case $\mathcal{M}= \mathbb{S}^3$ and $\mu$ is the volume.
 We  identify $\mathbb{S}^3$ as a subset of $\C^2$ and then define, for any $\alpha \in \T$, the map $\varphi_{\alpha}: \mathbb{S}^3 \to \mathbb{S}^3$,
\begin{equation}
\label{eq_per_action}
 \varphi_{\alpha}(z_1, z_2)=(e^{2\pi i \alpha}z_1, e^{2\pi i \alpha}z_2).
\end{equation}
The result that we prove is the following.
\begin{main}
\label{thm_main}
 For any $\Delta>1,\;  \varepsilon>0$ and $\alpha \in \T$, there exists a volume preserving and weakly mixing $F\in \text{Diff}^{\omega}_{\Delta}(\mathbb{S}^3))$ such that $\abs{F-\varphi_{\alpha}}_{\Delta}<\varepsilon.$
\end{main}

\begin{remark}
 The result holds as well for the case of any odd dimensional sphere excluding $\mathbb{S}^1$. The necessary changes in the proof for higher dimensional spheres are explained in Section \ref{sec_high}.
\end{remark}
A precise definition of $\text{Diff}^{\omega}_{\Delta}(\mathbb{S}^3)$ and the analytic topology are given in Section \ref{sec_prelim}.
In our construction $S_{\tilde \alpha}=\varphi_{\tilde \alpha}$ for some irrational number $\tilde \alpha\in \T$ that we can choose arbitrarily close to $\alpha$, 
but we have no control over its arithmetic properties. The proof requires that  $\tilde \alpha$ is well approximated by rational numbers in such a way that the necessary conditions for the  convergence of $F_n$ are satisfied, so we cannot guarantee $\tilde \alpha$ to be any arbitrarily chosen Liouville number.
This problem already appeared in \cite{AK} and was later solved in \cite{FS}, 
where the authors proved that $\tilde \alpha$ could be any 
Liouville number by using more accurate estimates. 
The specific scheme that we use to prove Theorem \ref{thm_main} builds on  the construction in \cite{FK}. 
In that case a finite sequence of transitive actions on the sphere are used to build the conjugacies $H_n$ for the \textit{Anosov-Katok} scheme in such a way that the limit map is uniquely ergodic with respect to the volume. 
Our goal is to show that their construction can be modified in order to obtain weakly mixing diffeomorphisms on odd dimensional spheres.

\subsection{Plan of the paper}

This work is divided in seven sections with the following content:
\begin{itemize}
 \item[i)] In Section \ref{sec_prelim} there is the necessary information concerning the topology of the space of maps that we consider, as well as the expression of the volume on the three sphere in appropriate coordinates. 

 In Section \ref{sec_decomp} we define two different types of decompositions into measurable subsets of the $3$-sphere which are used to verify that the approximating maps $\{F_n\}$ in \eqref{eq_intro} satisfy the properties that will lead to the limit map $F$ being weakly mixing. 

\item[ii)] Section \ref{sec_31} contains the definitions of \textit{approximate ergodicity} and \textit{approximate mixing}. Section \ref{sec_32} contains a lemma stating that if the maps of the sequence $\{F_n\}$ are approximately mixing and the sequence converges sufficiently fast, then the limit map is weakly mixing (see Lemma \ref{lemma_wkmix_criteria}).

\item[iii)] In Section \ref{sec_2} we explain the modification of the \textit{Anosov-Katok} scheme in \cite{FK} and give the explicit expressions of the maps involved. We also clarify why our modification is expected to yield the result claimed in Theorem \ref{thm_main}.  

\item[iv)] Section \ref{sec_proof} contains the proof of Theorem \ref{thm_main}, as well as the statement of the inductive proposition used to prove it, that is Proposition \ref{prop_dunaalaltra}.

\item[v)] In Section \ref{sec_unif_stre} stretching properties needed for the proof of Proposition \ref{prop_dunaalaltra} are defined and verified for certain maps used in our construction. Namely, we are referring to the map defined in equation \eqref{eq_phi} and it is verified in Lemma \ref{lemma_stretching}.

\item[vi)] Section \ref{sec_prop} contains the proof of Proposition \ref{prop_dunaalaltra}.

\item[vii)] Section \ref{sec_high} contains an indication on how to extend the proof to higher dimensional cases.
\end{itemize}

\subsection{Preliminaries and notation}
\label{sec_prelim}

As mentioned above, for the sake of simplicity we will restrict our study to the case of $\mathbb{S}^3$, although the proof can be extended to the case of any higher odd dimensional sphere up to minor modifications. In this section we introduce the basic properties that we use throughout the paper regarding the topology for the space of functions that we consider and the measure on the $3$-sphere.
\vspace{0.2cm}

\textbf{Analytic topology:} Let us consider $\mathbb{S}^{3}$ embedded in $\R^{4}$ and the standard complexification of $\R^{4}\subset \C^{4}$. Consider $f: \mathbb{S}^{3}\rightarrow \mathbb{S}^{3}$ such that each of its components is analytic. If the components of $f$ can be extended to  bounded holomorphic functions on the complex ball $B_{\Delta}=\{ z\in \C^{4}\; |\; \norm{z} < \Delta \}$ for a fixed $\Delta>1$ we write $f\in C^{\omega}_{\Delta}(\mathbb{S}^3)$. We will be considering the space of diffeomorphisms $\text{Diff}_{\Delta}^{\omega}(\mathbb{S}^3)\subset C^{\omega}_{\Delta}(\mathbb{S}^3)$, defined as the invertible maps in $C_{\Delta}^{\omega}(\mathbb{S}^3)$ such that their inverses also belong to $C_{\Delta}^{\omega}(\mathbb{S}^3)$. In this space we can define the distance between $f,g\in C^{\omega}_{\Delta}(\mathbb{S}^3)$ to be
\begin{equation*}
\abs{f-g}_{\Delta}= \max \left\{\sup_{z\in B_{\Delta}} \norm{f(z)-g(z)}, \; \sup_{z\in B_{\Delta}} \norm{f^{-1}(z)-g^{-1}(z)} \right\},
\end{equation*}
which makes $\text{Diff}^{\omega}_{\Delta}(\mathbb{S}^3)$ into a complete metric space.
 In case the components of such a function $f$ and its inverse can be extended to holomorphic functions in $\C^{4}$, we write  $f\in \text{Diff}^{\omega}_{\infty}(\mathbb{S}^3)$. For any $U\subset \C^{4}$, we define also
\begin{equation}
 \label{eq_norm_operador}
 \norm{\text{D}f}_{U}=\sup_{x\in U}\{ \max\{\norm{\text{D}f_x}, \norm{\text{D}f^{-1}_x}\}\},
\end{equation}
where the norms on the right hand side are the usual
operator norms
\begin{align*}
 & \norm{\text{D}f_x}=\sup\{\norm{\text{D}f_x v} \; |\; v\in \C^{4}, \; \norm{v}\leq 1\},\\
 & \norm{\text{D}f^{-1}_{x}}=\sup\{ \norm{\text{D}f^{-1}_x v} \; |\;   v\in \C^{4}, \; \norm{v}\leq 1\}.
\end{align*}

\vspace{0.2cm}

\textbf{Volume in Hopf coordinates:} We can identify $\mathbb{S}^{3}$ as a subset of $\mathbb{C}^2$, and write all points of the $3$-sphere in polar coordinates as $(z_1, z_2)=(r_1 e^{2\pi i \theta_1}, r_2 e^{2\pi i \theta_2})$, where the moduli for $z_1$ and $z_2$ satisfy the relation $r_1^2+r_2^2=1$, $0\leq r_1, r_2 \leq 1$. When using this notation it is convenient to parametrize $\mathbb{S}^3$ in Hopf coordinates,
\begin{align*}
 &z_1=\sin( \xi) e^{2\pi i \theta_1}, \\
 &z_2=\cos ( \xi)e^{2\pi i \theta_2},
\end{align*}
where $(\theta_1, \theta_2, \xi)\in \T^2 \times [0, \pi/2]$ and so
$r_1=\sin(\xi), \; r_2=\cos(\xi)$. 
Let us denote the parametrization of the $3$-sphere with such coordinates by $\psi$. Then a set $B\subset \mathbb{S}^3$ is defined to be measurable if 
 $\psi^{-1}(B)$ is Lebesgue measurable in $\T^2 \times [0, \pi/2]$, and we denote the induced $\sigma$-algebra on $\mathbb{S}^3$ by $\mathcal{B}$. The volume form induced by the standard euclidean metric in $\R^4$ is
 $dV= 4\pi^2 \sin(\xi) \cos(\xi) \; d\theta_1 \wedge d\theta_2 \wedge d\xi$, and hence the volume of any measurable set $B\in \mathcal{B}$ can be computed as 
\begin{equation*}
\mu(B)=4\pi^2 \int_{\psi^{-1}(B)}{ \sin(\xi) \cos(\xi) \; d\theta_1 \; d\theta_2 \; d\xi}.
\end{equation*}
Notice as well that we can compute the measure of any set $B\subset \mathbb{S}^3$ such that \[\psi^{-1}(B)=B_1 \times B_2 \subset \T^2\times [0, \pi/2]\] with $B_1$ and $B_2$ Lebesgue measurable as
\begin{equation}
\label{eq_measure}
\mu(B)=\bar \lambda(B_1)\cdot \mu_r(B_2), \; \text{where} \; \; \mu_r(B_2):=4\pi^2 \int_{B_2}{\sin (\xi) \cos(\xi) \; d\xi}
\end{equation}
and $\bar \lambda$ denotes the Lebesgue measure on $\T^2$.
For any Lebesgue measurable set $B\subset \T^2 \times [0, \pi/2]$, we will denote (if the projection is measurable)
$$\lambda^i(B):=\lambda(\Pi_{\theta_i} B), \quad i=1,2 $$
where $\lambda$ is the Lebesgue measure on $\T$.
Finally we denote by $\text{Diff}^{\omega}_{\infty}(\mathbb{S}^{3}, \mu)$ the subset of $\text{Diff}^{\omega}_{\infty}(\mathbb{S}^3)$ consisting of volume preserving maps.

 

\subsection{Decompositions into points}
\label{sec_decomp}
 The criterion that we use to obtain weak mixing, Lemma \ref{lemma_wkmix_criteria}, relies on the use of decompositions into points of the sphere. 
 Let us fix from this moment on that whenever we write $q\in \N$ it satisfies $q\geq 16$, and for any rational number $\alpha= p/q$, $p$ and $q$ are coprime and $q\geq 16$ as well.

%

\begin{definition} For each $n\in \N$, consider a collection of disjoint sets $\eta_n$ on $\mathbb{S}^3$. We say that $\{\eta_n\}$ converges to the decomposition into points (we also denote this by $\eta_n \to \varepsilon$) if for any  $B\in \mathcal{B}$, for any $n$ there exists $B_n \in \mathcal{B}$, which is a union of elements in $\eta_n$ and such that $\lim_{n \to \infty}\mu(B \Delta B_n)=0$ (here $\Delta$ denotes the symmetric difference). 
\end{definition}

In a slight abuse of notation we will from now on work with sets on $\mathbb{S}^3$ in Hopf coordinates without making explicit reference to the parametrization. Let us introduce the two types of decompositions into points that we use for $(\mathbb{S}^3, \mathcal{B}, \mu)$.

\begin{definition}
\label{def_coll_sets}
 Given $q\in \N$, $q\geq 16$, define the collection of sets 
 $\mathcal{C}_q := \{B_q^{k,l,m}\}_{k,l,m}$ given by 
 \begin{align*}
   B_q^{k,l,m}:=\{&(\theta_1,\theta_2, \xi )\in \T^2 \times[0, \pi/2]\; |  \;  \theta_1 \in (kq^{-1}, (k+1)q^{-1}), \; \theta_2\in (lq^{-1}, (l+1)q^{-1}), \\
   & \sin(\xi) \in (\sqrt{m}\; q^{-1/2}, \sqrt{(m+1)}\; q^{-1/2})  \} 
 \end{align*}
for $k,l,m\in \{1, \ldots, q-2\}$.
\end{definition}
 
It follows from \eqref{eq_measure} that 
any set $B\in \mathcal{C}_q$ satisfies $\mu(B)=2\pi ^2 q^{-3}$, hence
given any increasing sequence $\{q_n\}\subset \N$,
\begin{equation}
\label{eq_dense_collection}
 \mathcal{C}= \{\mathcal{C}_{q_n}\}_{n\in \N}
\end{equation}
is a decomposition into points.

We will need yet another type of decomposition into points that we describe next. First, for any $q\in \N, \; q\geq 16$ and $c\in [0,1)$, let us define the subsets 
\begin{equation}
 \label{eq_nqc}
N_{q,c}:=\left\{(\theta_1,\theta_2)\in \T^2, \; \theta_2=-(1-q^{-1})\theta_1+c,\;  \theta_1\in [0,1) \right\}
\end{equation}
and 
 \begin{equation}
\label{eq_xi_decomp}
\mathcal{F}_q:=\{\xi \in [0, \pi/2] \; |\; \sin(\xi)\in [q^{-1}, 1-q^{-1}]\}.
\end{equation}

\begin{definition}
\label{def_partial_q_decomposition}
 For any $q\in \N$, we say that $\eta_q$ is a partial $q$-decomposition if 
 \begin{equation}
\label{eq_decomp_into_points}
\eta_{q}:=\{I \times \{\xi \} \in \tilde{\eta}_{q,c}\times \mathcal{F}_q  , \; c \in [0,1) \},
\end{equation} 
where $\tilde{\eta}_{q,c}=\{I_i\}_{i\in \sigma(q,c)}$ with $I_i \subset N_{q,c}\; \forall i\in \sigma(q,c)$ and 
$\sigma(q,c)$ a finite index set such that all $I_i \in \tilde{\eta}_{q,c}$ are connected, $I_i \cap I_j = \emptyset$ for any $i,j \in \sigma(q,c)$ with $i\neq j$ and 
\begin{align}
  & \sum_{i\in \sigma(q,c)}{\lambda^1(I_i)}\geq 1- 3/\sqrt{q},\label{eq_first_cond_dec}\\
  & \lambda^1(I_i)\leq q^{-3} \label{eq_second_cond_dec}, \; \forall i\in \sigma(q,c).
 \end{align}
\end{definition} 
It follows from the definition above that for any sequence $q_n \to \infty$, $\{\eta_{q_n}\}$ with 
 $\eta_{q_n}$ a partial $q_n$-decomposition satisfies $\eta_{q_n} \to \varepsilon$.
Using a more specific type of decomposition into points of this type will be later required, whose specific form will be defined in Section \ref{sec_unif_stre}. We will motivate in that section the additional properties that the partial $q_n$-decompositions $\eta_{q_n}$ need to satisfy.

\section{Weak Mixing}
\label{sec_3}
 \subsection{Approximate ergodicity and mixing}
\label{sec_31}
In order to describe how to obtain weak mixing by using an AbC method, it will be useful to introduce definitions that reflect the idea
of being approximately ergodic and approximately mixing with respect to the decompositions into points described in Section \ref{sec_decomp}. For $q\geq 16$ and $\mathcal{C}_q$ as in Definition \ref{def_coll_sets}
let us define first, for every set $B_q^{k.l,m}\in \mathcal{C}_q$,  $k,l,m \in \{1, \ldots, q-2\}$ and $0<\varepsilon<1$, the two sets with $\varepsilon$-sized margins $B_{q,\pm \varepsilon}^{k,l,m}$ contained and containing $B_q^{k.l,m}$, formally defined by
{\small
\begin{align*}
   B_{q,\pm \varepsilon}^{k,l,m}:= \{(\theta_1,&\theta_2, \xi ) \in \T^2 \times[0, \pi/2]\; |\; \theta_1 \in (kq^{-1} \mp \varepsilon q^{-1}/16, (k+1)q^{-1}\pm \varepsilon q^{-1}/16), \\
   & \hspace{3.75cm} \theta_2 \in ( lq^{-1} \mp \varepsilon q^{-1}/16, \; (l+1)q^{-1} \pm \varepsilon q^{-1}/16), \\
   & \hspace{1.3cm} \sin(\xi) \in (\sqrt{m/q}\;  \mp \sqrt{\varepsilon/(16q)}, \sqrt{(m+1)/q}\;  \pm \sqrt{\varepsilon/(16q)} \;)\}.
 \end{align*}
 }
 \hspace{-0.3cm} For the sake of simplicity we will rename the sets above, for any $B= B^{k,l,m}_q\in \mathcal{C}_q$, simply as $B_\varepsilon:=B^{k,l,m}_{q,\varepsilon}$.
 The reason why we introduce these sets is explained in Remark \ref{rem_why_sets}.
 Let us state some useful inequalities in  the form of the 
 following lemma. 
 \begin{lemma}
 \label{lemma_distancies}
    There exists a constant $\gamma >0$ such that for all $q\in \N$ and $\varepsilon>0$ sufficiently small, for any $B\in \mathcal{C}_q$, the sets $B_{\pm \varepsilon}$ satisfy the following inequalities:  \begin{align}
  &\mu(B_\varepsilon \Delta B_{-\varepsilon})<2\varepsilon \mu(B), \label{sim_diff_1}\\
  &\mu(B \Delta B_{\pm \varepsilon}) <\varepsilon \mu(B) \label{sim_diff_2},\\
  & \text{dist}(\partial B_{-\varepsilon}, \partial B) \geq \gamma\frac{\varepsilon}{q^{3/2}}, \label{eq_dis_bs1} \\
  &\text{dist}(\partial B, \partial B_{\varepsilon})\geq \gamma \frac{\varepsilon}{q^{3/2}}. \label{eq_dis_bs2}
  \end{align}
 \end{lemma}

 \begin{proof}
  The proof follows by using simple estimates to bound the distance and measure using \eqref{eq_measure}.
 \end{proof}

\begin{definition}[Approximate ergodicity]
\label{def_approx_ergodic}
 Given $\varepsilon >0, N\in \N$,  we say that a $\mu$-preserving map $f$ is
 $(q, \varepsilon, N)$-ergodic if for every $B\in \mathcal{C}_q$,
 $n \geq N$ and $x\in \mathbb{S}^3$
 \begin{equation}
 \label{eq_er}
  \bigg{|}\frac{1}{n}\sum_{k=0}^{n-1}{\mathbbm{1}_{B_{\pm \varepsilon}}(f^k(x))- \mu(B_{\pm \varepsilon})}\bigg{|}< 3\varepsilon \mu(B).
 \end{equation}
\end{definition}
\vspace{0.2cm}
Given  $\eta_{q'}$ a partial $q'$-decomposition with $q'>q$, $q'\in \N$ we can also define an analogous concept for weak mixing.
\begin{definition}[Approximate mixing]
\label{def_approx_mix}
 Given $H\in \text{Diff}^{\omega}_{\infty}(\mathbb{S}^3, \mu)$, $\varepsilon>0$ and $m\in \N$, a $\mu$-preserving map $F$ is
 $(q, q', \varepsilon, m,H)$-mixing if for every $B\in \mathcal{C}_q$ and $I\in \eta_{q'}$
\begin{equation}
 \bigg{|}\frac{\lambda^1(I \cap H(F^{-m}(B_{\pm \varepsilon})))}{\lambda^1(I)}-\mu(B_{\pm \varepsilon}) \bigg{|} \leq 30 \varepsilon \mu(B). \label{eq_crit_wmix2}
\end{equation}
\end{definition}

\begin{remark}
 It would be more intuitive to refer to the two properties above as approximate unique ergodicity and approximate weak mixing. We did not do this only in order to keep the notation short.
\end{remark}

\vspace{0.2cm}

\subsection{Criterion for weak mixing}
\label{sec_32}
 Assume from now on that $\Delta>1$ is fixed.
 In this section we present the proof of the criterion to prove weak mixing, that is Lemma \ref{lemma_wkmix_criteria}. In the statements we fix a decomposition into points
 $\{\eta_n\}$, where $\eta_{n}=\eta_{q'_n}$ is a partial $q'_n$-decomposition and $\{q'_n\}\subset \N$ is such that $q'_n\to \infty$. We also fix $\mathcal{C}= \{\mathcal{C}_n\}_{n\in \N}$ with 
 $\mathcal{C}_n=\mathcal{C}_{q_n}$ as in \eqref{eq_dense_collection}, also with $\{q_n\}\subset \N$ such that $q_n\to \infty$.  
 
\begin{lemma}
\label{lemma_wkmix_crit}
Consider $\{\eta_n\}$ and $\mathcal{C}$ as described above. 
 Assume that we have a $\mu$-preserving map $F\in C^{\omega}_{\Delta}(\mathbb{S}^3)$ and a sequence $\{H_n\}\subset \text{Diff}^{\omega}_{\infty}(\mathbb{S}^3, \mu)$ 
 such that $\nu_n= \{H^{-1}_n(I), \; I \in \eta_n\}$ is a decomposition into points. 
 If there exists an increasing sequence $\{m_n\}\subset \N$ such that for every $\varepsilon>0$, for $n$ large enough we have that for all $\Gamma \in \nu_n$ and $B\in \mathcal{C}_{n}$ 
 \begin{equation}
 \label{eq_crit_wmix}
  \abs{\lambda_n(\Gamma \cap F^{-m_n}(B))-\lambda_n(\Gamma)\mu(B)} \leq \varepsilon \lambda_n(\Gamma)\mu(B),
 \end{equation}

 where $\lambda_n(\Gamma):=\lambda^1( H_n(\Gamma)) $, then $F$ is weakly mixing.
\end{lemma}

\begin{proof}
 For  a fixed $\varepsilon'>0$, choose $n$ large enough so that we can approximate any measurable sets $A,B \in \mathcal{B}$ by $B'$ a finite union of elements of $\mathcal{C}_n$ with $\mu(B' \Delta B)< \varepsilon'/8$ and
 \begin{equation*}
   A'= \bigcup_{\substack{c\in [0,1) \\ \xi \in \mathcal{F}_{q'_n}}} \bigcup_{i\in \sigma(\xi,c)}{\Gamma_i},
 \end{equation*}
 where $\sigma(\xi,c)$ is a finite index set such that for all $i\in \sigma(\xi,c)$, $H_n(\Gamma_i)= I_i \in \tilde \eta_{q'_n,c}\times\{\xi\}\subset \tilde{\eta}_{q'_n}\times \mathcal{F}_{q'_n}=\eta_{q_n'}$ and 
 $\mu(A'\Delta A)< \varepsilon'/8$. It follows that
\begin{align}
  \abs{\mu(A\cap F^{-m_n}(B))-\mu(A)\mu(B)} 
  &\leq \mu((A \Delta A')\cap F^{-m_n}(B)) \nonumber \\
  &+\mu(A\cap F^{-m_n}(B \Delta B')) \nonumber\\
   &+\abs{\mu(A' \cap F^{-m_n}(B'))-\mu(A')\mu(B')} \nonumber\\
 &+\mu(A)\mu(B \Delta B') + \mu(B')\mu(A \Delta A') \nonumber\\
   &\leq \abs{\mu(A'\cap F^{-m_n}(B'))-\mu(A')\mu(B')}+\varepsilon'/2. \label{eq_pre_demo}
\end{align}
Therefore due to the assumptions of the lemma, for $\varepsilon= \varepsilon'/(8\pi^4)$ we can consider $n\geq N$ sufficiently large such that
\eqref{eq_crit_wmix} and \eqref{eq_pre_demo}  hold.
It can then be seen, using the volume preserving change of variables 
defined by
 $$
 \left\{
 \begin{array}{l}
   \theta_1=\theta_1', \\ 
   \theta_2 = \theta_2'-(1-{q'}_n^{-1})\theta_1',\\ 
   \xi=\xi',
 \end{array}
 \right.
 $$
 which essentially sends the elements in $\eta_{q'_n}$ to horizontal segments, that
\begin{align*}
  &\abs{\mu(A'\cap F^{-m_n}(B'))-\mu(A')\mu(B')}\\
  &=\abs{\mu(H_n(A')\cap H_n(F^{-m_n}(B')))-\mu(H_n(A'))\mu(B')} \\
  &= \bigg{|}\int_{\xi' \in \mathcal{F}_{q'_n} }\int_{\theta_2'\in[0,1)}
  {\sum_{i\in \sigma(\xi', \theta_2')}(\lambda^1(I_i \cap H_n(F^{-m_n}(B')))-\lambda^1(I_i)\mu(B'))\;d\theta_2' \;  d\mu_r} \bigg{|}\\
  &\leq \int_{\xi' \in \mathcal{F}_{q'_n} }\int_{\theta_2'\in[0,1)}
  {\sum_{i\in \sigma(\xi', \theta_2')}\abs{\lambda_n(\Gamma_i \cap F^{-m_n}(B'))-\lambda_n(\Gamma_i)\mu(B')}\;d\theta_2' \; d\mu_r}\\
  &\leq \int_{\xi' \in \mathcal{F}_{q'_n} }\int_{\theta_2'\in[0,1)}
  {\sum_{i\in \sigma(\xi', \theta_2')}\varepsilon \lambda_n(\Gamma_i)\mu(B')\; d\theta_2' \: d\mu_r}= \varepsilon \mu(A')\mu(B')<\varepsilon'/2,
\end{align*}
which together with \eqref{eq_pre_demo} finishes the proof.
\end{proof}

\begin{remark}
\label{rem_why_sets}
Since the \textit{Anosov-Katok} scheme does not give an explicit expression for the limit map $F$, we need to find a similar criterion to the one in Lemma \ref{lemma_wkmix_crit} which can be verified instead for the approximating maps $\{F_n\}$ and the sets $B_{\pm \varepsilon}$. This allows us to recover, if $\{F_n\}$ converges sufficiently fast, the original assumption in Lemma \ref{lemma_wkmix_crit}. It will suffice that the approximating maps converge sufficiently fast and are approximately mixing in the sense of Definition \ref{def_approx_mix}.
\end{remark}

Consider $\gamma>0$ as given by Lemma \ref{lemma_distancies}.

\begin{lemma}[Criterion for weak mixing]
\label{lemma_wkmix_criteria}
 Consider $\{\eta_n\}$ and $\mathcal{C}$ as described at the beginning of the section. 
 Suppose that $F$ is the limit diffeomorphism of $\{F_n\} \subset \text{Diff}_{\Delta}^{\omega}(\mathbb{S}^3)$
 and that there are sequences $\varepsilon_n \to 0$ and $m_n\to \infty$ such that for all $n\in \N$ the map $F_n$ is $(q_n, q_n', \varepsilon_n, m_n, H_n)$-mixing, where $\{H_n\}\subset \text{Diff}^{\omega}_{\infty}(\mathbb{S}^3, \mu)$ is such that the collections of sets ${\nu_n= \{H^{-1}_n(I), \; I \in \eta_n\} }$ converge to the decomposition into points. 
 If for all $n$ sufficiently large
 \begin{equation}
 \label{eq_fast_mixing}
 \abs{F^{m_n}-F^{m_n}_n}_{\Delta}< \gamma \varepsilon_n q_{n}^{-3/2},
 \end{equation}
then $F$ is weakly mixing. 
\end{lemma}

\begin{proof}
 We only need to see that the assumptions of Lemma \ref{lemma_wkmix_criteria} imply the assumptions of Lemma \ref{lemma_wkmix_crit}. Fix $\varepsilon>0$. We can consider $n$ big enough so that $\varepsilon_n\leq  \varepsilon/90$ and the map $F_n$
 is $(q_n, q_n' , \varepsilon_n, m_n, H_n)$-mixing.
 It follows from Lemma \ref{lemma_distancies} that for any $B \in \mathcal{C}_n$, we have
 \begin{equation}
 \label{eq_coses}
  \mu(B \Delta B_{\pm \varepsilon_n})\leq \varepsilon_n \mu(B).
 \end{equation} 
 According to \eqref{eq_dis_bs1}, \eqref{eq_dis_bs2} and \eqref{eq_fast_mixing} we obtain, for any $\Gamma \in \nu_n$,
 \begin{equation*}
 \Gamma \cap F_n^{-m_n}(B_{-\varepsilon_n})\subset \Gamma \cap  F^{-m_n}(B) \subset \Gamma \cap F_n^{-m_n}(B_{\varepsilon_n}).
 \end{equation*}
 Hence, since \[\lambda_n(\Gamma \cap  F_n^{-m_n}(B_{\pm \varepsilon_n}))=\lambda^1(I \cap H_n(F_n^{-m_n}(B_{\pm \varepsilon_n})))\] and \[\lambda_n(\Gamma)= \lambda^1(H_n(\Gamma))= \lambda^1(I),\]  it follows from \eqref{eq_crit_wmix2} that 
 \begin{equation*} 
  \lambda_n(\Gamma)\mu(B_{-\varepsilon_n})-\varepsilon/3 \lambda_n(\Gamma)\mu(B)\leq \lambda_n(\Gamma \cap F^{-m_n}(B)) \leq (1+\varepsilon/3) \lambda_n(\Gamma)\mu(B_{\varepsilon_n}),
 \end{equation*}
 which due to \eqref{eq_coses} implies
 \begin{equation*}
  \abs{\lambda_n(\Gamma \cap F^{-m_n}(B))-\lambda_n(\Gamma )\mu(B)} \leq \varepsilon \lambda_n(\Gamma)\mu(B).
 \end{equation*}
 This finishes the proof.
\end{proof}


\section{Construction of the map}
\label{sec_2}
The \textit{Anosov-Katok} scheme that we use can be understood as a modification of the scheme used in \cite{FK}. 
The main result in \cite{FK} (see Theorem $1$) states that for any $\beta \in \T$ there exists a uniquely ergodic and volume preserving diffeomorphism $f\in \text{Diff}^{\omega}_{\Delta}(\mathbb{S}^3)$, arbitrarily close to $\varphi_{\beta}$ in the analytic topology, which is obtained as the limit of a sequence $\{f_l\}$ with $f_{-1}=\varphi_{\beta_0}$ and for $l\geq 0$
\begin{equation*}
 f_l=H_l^{-1} \circ \varphi_{\beta_{l+1}} \circ H_l,
\end{equation*}
with $H_l \in \text{Diff}^{\omega}_{\infty}(\mathbb{S}^3, \mu)$ defined inductively as $H_{-1}=id$,
\begin{equation*}
H_l= h_l \circ H_{l-1}, \; \quad h_l \circ \varphi_{\beta_{l}}= \varphi_{\beta_l} \circ h_l
\end{equation*}
for appropriate sequences $\{\beta_l\}_{l\geq 0} \subset \Q$ and $\{h_l\}_{l\geq 0} \subset \text{Diff}^{\omega}_{\infty}(\mathbb{S}^3, \mu)$. We can then reformulate their conclusion in the form of the following proposition.
\begin{proposition}
\label{prop_unique_erg}
 For any $\beta \in \T$, $\varepsilon>0$ and $\Delta>1$ there exist
 sequences $\{\beta_l\}_{l\geq 0}\subset \Q $ and $\{h_l\}_{l\geq 0}\subset \text{Diff}^{\omega}_{\infty}(\mathbb{S}^{3}, \mu)$ such that 
 \begin{itemize}
  \item[i)] $h_l^{-1} \circ \varphi_{\beta_l} \circ h_l= \varphi_{\beta_l}$,  for any $l \geq 0$.
  \item[ii)] For $f_l:= h_0^{-1}\circ \ldots \circ h_{l}^{-1}\circ \varphi_{\beta_{l+1}} \circ h_l \circ \ldots \circ h_0$, the map $f:=\lim_{l\to \infty} f_l \in \text{Diff}^{\omega}_{\Delta}(\mathbb{S}^3)$ is a uniquely ergodic volume preserving map and $\abs{f-\varphi_{\beta}}_{\Delta}<\varepsilon$.
 \end{itemize}
\end{proposition}

\vspace{0.2cm}
 
\begin{remark}
The result in Proposition \ref{prop_unique_erg} holds for odd dimensional spheres of dimension greater or equal to three.
\end{remark}

Let us introduce the specific form of the map that we will use to modify the construction in Proposition \ref{prop_unique_erg}.
For any $q\in \N$ and $s\in \T$, we define first 
$\zeta^s_{q}: \mathbb{S}^3 \rightarrow \mathbb{S}^3$ by
\begin{equation}
 \label{eq_funcio}
\zeta^s_{q}(z):=(e^{2\pi i s}z_1, e^{2\pi i (1+q^{-1})s}z_2),
\end{equation}

and $\chi_q: \mathbb{S}^3 \rightarrow \R$ by
\begin{align}
\chi_{q}(z)&:=\text{Re}(z_1^{q(q +1)})\text{Re}(z_2^{q^2})+
\text{Im}(z_1^{q(q +1)})\text{Im}(z_2^{q^2}) \label{eq_nu}\\
&=r_1^{q(q+1)} r_2^{q^2} \cos(2\pi q^2[\theta_1(1+ q^{-1})-\theta_2]) \nonumber.
\end{align}

Notice that $\chi_q$ is entire in $(x_1, y_1, x_2, y_2)$. 
This is so because the functions in \eqref{eq_nu} are polynomials in
the variables $x_1$, $y_1$, $x_2$, $y_2$, (where $z_1=x_1+iy_1$ and $z_2=x_2+iy_2$), and thus they are real entire.
The maps that we use to modify the construction of Proposition
\ref{prop_unique_erg} are of the form
\begin{equation}
\label{eq_g}
g_{q,A}(z)=\zeta_{q}^{A\chi_{q}(z)}(z), \;  A>0.
\end{equation}

Since 
$\chi_q$ is real entire, it follows from \eqref{eq_funcio} that 
$g_{q,A}$ is also entire. 
 The map $g_{q,A}$ is expressed in Hopf coordinates as
 \begin{align}
  & g_{q, A}(\theta_1, \theta_2, \xi)= (\theta_1+ A r_1^{q(q+1)} r_2^{q^2} \cos(2\pi q^2[\theta_1(1+ q^{-1})-\theta_2]), \nonumber \\
  & \quad \quad \theta_2+ (1+q^{-1}) A r_1^{q(q+1)} r_2^{q^2} \cos( 2\pi q^2[\theta_1(1+ q^{-1})-\theta_2]), \xi). \label{eq_g_hopf}
 \end{align}
 
 It follows from a computation that $\det(\text{D}g_{q, A})=1$ and since $g_{q,A}$ is the identity on the third coordinate, it preserves $\mu$.
 Notice also that $g_{q,A}^{-1}=g_{q, -A}$ and thus $g_{q,A} \in \text{Diff}^{\omega}_{\infty}(\mathbb{S}^3, \mu)$, and also that
 $g_{q,A}\circ \varphi_{p/q} = \varphi_{p/q} \circ g_{q,A}$ (this commuting property is essential, as we will see later, for the conjugacies in AbC constructions). 
 Let us define, for any $\tilde{\alpha} \in \T, q\in \N$ and $A>0$, the map
\begin{equation}
\label{eq_phi}
 \Phi_{q, A, \tilde \alpha}= 
 g_{q,A}^{-1} \circ \varphi_{\tilde \alpha} \circ g_{q,A}.
\end{equation}
 In the next section we show how $\Phi_{q, A, \tilde \alpha}$ plays a key role in the modification of the scheme in \cite{FK}, by explaining the idea for one step of our construction.  
 
 \vspace{0.2cm}

\subsection{Plan for one step of the construction}
\label{sec_idea}
Assume that we are given a volume preserving map $F= V^{-1} \circ \varphi_{\alpha} \circ V$, $V\in \text{Diff}^{\omega}_{\infty}(\mathbb{S}^3, \mu)$, $\alpha=p/q \in \Q$. 
Let us explain how to use Proposition \ref{prop_unique_erg} to obtain an approximately mixing map lying arbitrarily close to $F$.
\begin{itemize}
 \item[1)] For an arbitrary $\varepsilon>0$, since the limit map $f$ in Proposition \ref{prop_unique_erg} is uniquely ergodic, there exists $l$ big enough such that 
$f_l=H_l^{-1} \circ \varphi_{\beta_{l+1}} \circ H_l$ with $\beta_{l+1}=p_{l+1}/q_{l+1}$ and $H_l= h_{l} \circ \ldots \circ h_0$ is ``approximately ergodic'' in the following sense: for a sufficiently large number of iterations the orbits of the map $f_{l}$ are uniformly distributed among the elements of $\mathcal{C}_q$ with an ``$\varepsilon$-precision'' (recall Definition \ref{def_approx_ergodic} for a formal description). Since the map $V$ is volume preserving, we can also assume that this is the case for the conjugate map
$f'_l:= V^{-1}\circ H_l^{-1} \circ \varphi_{\beta_{l+1}} \circ H_l \circ V$, which can be assumed to be arbitrarily close to $V^{-1}\circ \varphi_{\beta_0}\circ V$.
By choosing $\beta_{0}$ to be very close to our original $\alpha$, the distance between the initial $F$ and $f'_l$ can then be made arbitrarily small. 
Notice next  that for any $x\in \mathbb{S}^3$, $y:= V^{-1}\circ H^{-1}_l(x)$ satisfies that the set
$\mathcal{O}_y=\{\varphi^n_{\beta_{l+1}}(H_l \circ V (y)) \}^{\infty}_{n=0}$ is $1/q_{l+1}$-dense in the set
\[D_{\xi,c}:=\{\theta_2=\theta_1+c, \theta_1\in \T \} \times \{\xi\}\] 
for some $0\leq c<1$, $\xi \in [0, \pi/2]$,
and it follows from the approximate ergodicity of $f'_l$ that the map $V^{-1} \circ H_l^{-1}$ must distribute uniformly the points of $\mathcal{O}_y$ among the sets in $\mathcal{C}_q$ with an ``$\varepsilon$-precision''. 

\item[2)] In order to use the previous step to obtain approximate weak mixing, we introduce the map 
$g_{q_{l+1}, A}$ defined in \eqref{eq_g} as we describe next. We prove in Proposition \ref{lemma_stretching} that it is possible to find constants $A>0$, $m\in \N$, $\tilde{\alpha}\in \Q$ and a partial $q_{l+1}$-decomposition $\eta_{q_{l+1}}$ such that the $m$th iterate of the map $ \Phi_{q_{l+1}, A, \tilde \alpha}= 
 g_{q_{l+1},A}^{-1} \circ \varphi_{\tilde \alpha} \circ g_{q_{l+1},A}$
 stretches all $I \in \eta_{q_{l+1}}$ uniformly in measure (in a sense that will be made precise later in Definition \ref{def_stretch}) onto a curve lying close to $D_{\xi, c}$ for some $c\in [0,1)$,  $\xi \in [0, \pi/2]$. This means that for any  ${\Gamma \in V^{-1} \circ H_l^{-1}(\eta_{q_{l+1}})}$
\begin{equation}
\label{eq_ini_stre}
  \Phi_{q_{l+1}, A, \tilde \alpha}^m  \circ H_{l}\circ V (\Gamma)\sim D_{\xi, c}
\end{equation}
for some $0\leq c <1$, $\xi \in [0, \pi/2]$.
 As a consequence, in the same way as the points of the set $\mathcal{O}_{y} \subset D_{\xi, c}$ were distributed uniformly among the sets in $\mathcal{C}_q$ when applying $V^{-1} \circ H^{-1}_l$, now the $m$th iteration of the map
$$
 \bar{F}:=V^{-1}\circ H_l^{-1} \circ  \Phi_{q_{l+1}, A, \tilde \alpha} \circ H_l \circ V
$$
will distribute the measure of every segment $\Gamma \in V^{-1}\circ H_l^{-1}(\eta_{q_{l+1}})$ among the sets in $\mathcal{C}_q$ proportionally to their measure again with an ``$\varepsilon$-precision'', or more formally, it will be $(q, q_{l+1}, \varepsilon, m, H_l \circ V)$-mixing. 
 At the same time, an appropriate choice of the constants can be made so that the distance between $\bar F$ and $F$ remains small.
\end{itemize}
By iterating the two steps above we can construct a sequence of maps $\{F_n\}$ with each $F_n$ being approximately mixing and converging sufficiently fast so that the limit is weakly mixing by Lemma \ref{lemma_wkmix_criteria}.

\section{Proof of Theorem \ref{thm_main}}
\label{sec_proof}
In this section we finish the proof of Theorem \ref{thm_main}. Before doing this, we will need a couple of preliminary results.

\begin{lemma}
 \label{lemma_canvi_decomp}
 Let $\varepsilon>0$, $q\in \N$ and $K \in \text{Diff}^{\omega}_{\infty}(\mathbb{S}^3, \mu)$ be fixed. For any integer $q'>q$ sufficiently large, for any partial $q'$-decomposition $\eta_{q'}$ ,  
 $\nu_{q'} := \{ K^{-1}(I) \; |\; I \in \eta_{q'}\}$ satisfies that
 for any $B \in \mathcal{C}_q$ there exists 
 $\{\Gamma_i\}_{i\in \sigma} \subset \nu_{q'}$ with 
 $\mu(B \Delta \cup_{i\in \sigma}{\Gamma_i})< \varepsilon$.
\end{lemma}

\begin{proof}
Consider the collection of sets $K(B), \; B\in \mathcal{C}_q$. Since $\{\eta_{q'}\}$ converges to the decomposition into points as $q'\to \infty$, there exists $q'$ sufficiently large such that for any 
$B \in \mathcal{C}_q$ there exists a collection $\{I_i\}_{i\in \sigma} \subset \eta_{q'}$ with
$\mu(K(B)\Delta \cup_{i\in \sigma} I_i)<\varepsilon$. Since $K$ is a volume preserving diffeomorphism, this finishes the proof.
\end{proof}
The following proposition is the core of the iterative scheme needed for the proof of Theorem \ref{thm_main}. 
\begin{proposition}[Inductive step]
 \label{prop_dunaalaltra}
 Given $V\in \text{Diff}^{\omega}_{\infty}(\mathbb{S}^3, \mu)$, $\alpha= p/q \in \Q$,
 $\varepsilon>0$ and $m\in \N$, there exists $H\in \text{Diff}^{\omega}_{\infty}(\mathbb{S}^3, \mu)$ such that for any $L>0$ there are
   $m<\tilde m \in \N$,  $A>0$, $\tilde{\alpha}=\tilde p/\tilde q \in \Q$, $q<q'<\tilde{q}$ with
   \begin{equation}
   \label{eq_fer_dist}
    q'>q+L
   \end{equation}
   and a partial $q'$-decomposition $\eta_{q'}$ 
  such that
  $$F:= V^{-1}\circ {H}^{-1} \circ \Phi_{q', A, \tilde{\alpha}}  \circ {H} \circ V$$
  with $\Phi_{q', A, \tilde{\alpha}}$ defined by \eqref{eq_phi} satisfies: 
 \begin{itemize}
  \item[i)] $\abs{F^j-V^{-1}\circ \varphi^j_{\alpha}\circ V}_{\Delta}<\varepsilon, \; \; 0 < j \leq m$,
  \item[ii)] $F$ is $(q, q', \varepsilon, \tilde m, H\circ V)$-mixing.
 \end{itemize}
\end{proposition}
Let postpone the proof of Proposition \ref{prop_dunaalaltra} to Section \ref{sec_prop} and focus now instead on the proof of Theorem \ref{thm_main}.
\begin{proofof}{Theorem \ref{thm_main}}
Recall that, for a fixed $\Delta>1$, $\bar \varepsilon>0$ and $\bar \alpha \in \T$ we want to prove that there exists a volume preserving map 
$F\in \text{Diff}^{\omega}_{\Delta}(\mathbb{S}^3)$ which is weakly mixing and satisfies $ \abs{F- \varphi_{\bar \alpha}}_{\Delta}< \bar\varepsilon$.
We claim that $F$ can be found as the limit of a sequence 
$\{F_n\}_{n\geq 0}\subset \text{Diff}^{\omega}_{\Delta}(\mathbb{S}^3)$ 
of the form, for $n\geq 1$,
\begin{equation}
\label{eq_int}
 F_n= H_{n-1}^{-1} \circ \Phi_{q'_{n-1}, A_{n-1}, \alpha_{n}}
 \circ H_{n-1}, 
 \; \text{with} \; H_{n-1}\in \text{Diff}^{\omega}_{\infty}(\mathbb{S}^3, \mu)
\end{equation}
for an appropriate choice of parameters.
Let us show that such a sequence can be built by induction.
Consider first $\alpha_0= p_0/q_0\in \Q$ with
\begin{equation}
\label{eq_diff_inicial}
 \abs{\varphi_{\bar \alpha}- \varphi_{\alpha_0}}_{\Delta}<\bar \varepsilon/4,
\end{equation}
and let $F_0:= \varphi_{\alpha_0}$. We apply Proposition \ref{prop_dunaalaltra} for $V=Id$,  $\alpha=\alpha_0$,  $\varepsilon=\varepsilon_0 = \min \{\bar \varepsilon/4, \gamma \bar{\varepsilon}/4\} $ and $m=1$ to obtain $H_0:=H$ such that, for
$L_0(\varepsilon_0, q_0, H_0)$
sufficiently large for the conclusion of Lemma \ref{lemma_canvi_decomp} (for fixed $\varepsilon=\varepsilon_0/\#\mathcal{C}_{q_0}, q=q_0, K=H_0$) to be satisfied for any \[q'>q_0+L_0(\varepsilon_0, q_0, H_0),\] 
we have
\[A_0:=A,\quad m_0:=\tilde m,\quad  
\alpha_1=p_1/q_1:= \tilde \alpha,\quad q_0+L_0(\varepsilon_0, q_0, H_0)<q_0'<q_1\] and a partial  $q_0'$-decomposition $\eta_{q_0'}$ such that
\begin{equation}
 \label{eq_first_diff}
\abs{F_1-F_0}_{\Delta}<\varepsilon_0
\end{equation}
and 
$F_1$ is $(q_0, q_0', \varepsilon_0, m_0, {H}_0)$-mixing.\newline 
For $n\geq 2$ we can apply again Proposition \ref{prop_dunaalaltra} for $V=g_{q'_{n-2},\; A_{n-2}} \circ H_{n-2} \circ \ldots \circ g_{q_0',\; A_0} \circ H_0$, $\alpha=\alpha_{n-1}$,
\begin{equation}
\label{eq_cond_epsilon}
\varepsilon_{n-1} =2^{-1} q_{n-2}^{-3/2} \min \left\{\varepsilon_{n-2}, \; \gamma \varepsilon_{n-2} \right\}
\end{equation}
and $m=m_{n-2}$ to obtain $H$ such that, for
$L(\varepsilon_{n-1}, q_{n-1}, H\circ V)$ sufficiently large for the conclusion of Lemma \ref{lemma_canvi_decomp}  (for fixed $\varepsilon=\varepsilon_{n-1}/\# \mathcal{C}_{q_{n-1}}$, $q=q_{n-1}$, $K=H_{n-1}:=H\circ V$)  to be satisfied for any 
\[q'>q_{n-1}+L_{n-1}(\varepsilon_{n-1}, q_{n-1}, H\circ V),\]
we have
\[A_{n-1}:=A, \quad m_{n-1}:=\tilde m, \quad \alpha_n:=\tilde \alpha,\quad q_{n-1}+L_{n-1}(\varepsilon_{n-1}, q_{n-1}, H\circ V) <q_{n-1}'<q_n, \] 
and a partial $q_{n-1}'$-decomposition $\eta_{q_{n-1}'}$ such that 
\begin{equation}
 \label{eq_good_dist}
 \abs{F^j_n-F^j_{n-1}}_{\Delta}<\varepsilon_{n-1}, \; \text{for all} \; \;  0\leq  j \leq m_{n-2}
\end{equation}
and $F_n$ is $( q_{n-1}, q_{n-1}', \varepsilon_{n-1},  m_{n-1}, H_{n-1})$-mixing. This finishes the construction of the sequence.
Notice that it follows from Proposition \ref{prop_dunaalaltra} that the sequence $\{m_{n}\}$ is strictly increasing. 
Let us now justify that the limit of $\{F_{n+1}\}_{n\geq 1}$ exists and satisfies the conclusions of Theorem \ref{thm_main}.
It is clear from \eqref{eq_cond_epsilon} and \eqref{eq_good_dist} that $\{F_{n+1}\}_{n\geq 1}$ is a Cauchy sequence, and hence $F=\lim_{n\to \infty} F_n\in \text{Diff}^{\omega}_{\Delta}(\mathbb{S}^3)$. Also $F$ is volume preserving as it is the limit of volume preserving maps. \newline 
As a consequence of our choice of the constants $L_{n-1}(\varepsilon_{n-1}, q_{n-1}, H\circ V)$ for $n\geq 1$ and Lemma \ref{lemma_canvi_decomp} we have that the union of elements in $\nu_{n-1}=\{ H^{-1}_{n-1}(I) \; | \; I \in \eta_{q'_{n-1}}\}$ can approximate up to an $\varepsilon_{n-1}$ error in measure any union of elements in $\mathcal{C}_{q_{n-1}}$. Since $\varepsilon_{n-1}\to 0$ and $\mathcal{C}_{q_{n-1}}\to \varepsilon$, it follows that $\nu_{n-1}$ converges to the decomposition into points as well.\newline
It follows from \eqref{eq_cond_epsilon} that for all $n\geq 1$
\begin{equation}
\label{eq_int1}
\abs{F^{m_{n-1}}- F^{m_{n-1}}_{n}}_{\Delta}\leq \gamma \varepsilon_{n-1} q_{n-1}^{-3/2}.
\end{equation}
Thus the sequence $\left\{F_{n+1}\right\}_{n\geq 1}$ satisfies equation \eqref{eq_fast_mixing} for every $n\geq 1$. Since clearly $\varepsilon_{n}\to 0$ and $\{m_{n}\}$ is strictly increasing the assumptions  of Lemma \ref{lemma_wkmix_criteria} are satisfied and thus $F$ is weakly mixing.
Finally due to \eqref{eq_diff_inicial}, \eqref{eq_first_diff}, \eqref{eq_good_dist} and \eqref{eq_cond_epsilon}  we obtain that 
$\abs{F-\varphi_{\bar \alpha}}_{\Delta}<\bar \varepsilon$ and this finishes the proof.
\end{proofof}

\section{Uniform stretching}
\label{sec_unif_stre}
Before continuing with the proof of Proposition \ref{prop_dunaalaltra}, we need to introduce the definition of \textit{uniform stretching}.
\begin{definition}[Uniform stretching] 
\label{def_stretch} 
Given $\varepsilon>0$ and 
$k>0$, we say that a real continuous function $f$ on an interval $I\subset \R$ is $(\varepsilon, k)$-uniformly stretching on $I$ if for $J=[\; \inf_{I}f, \; \sup_{I}f]$
\begin{equation*}
 \lambda(J)\geq k
\end{equation*}
and for any interval $\tilde{J}\subset J$ we have 
\begin{equation*}
 \bigg{|}\cfrac{\lambda(I \cap f^{-1}(\tilde{J}))}{\lambda(I)}-\cfrac{\lambda(\tilde{J})}{\lambda(J)} \bigg{|} \leq \varepsilon \cfrac{\lambda(\tilde{J})}{\lambda(J)} \;.
\end{equation*}
\end{definition}
The main idea behind \textit{uniform stretching} is that the interval $I$ is stretched ``almost linearly'' in measure. 
The following lemma, which can be found in \cite{Fayad} (see Lemma 2), provides a criterion for a function to be uniformly stretching.

\begin{lemma}\label{lemma_stretch_modified} If $f\in C^{2}(\R)$ is monotonic on an interval $I \subset \R$ and 
 \begin{align*}
   & \lambda(J)\geq k, \\
   & \textit{sup}_{I}\;\abs{f''(x)} \lambda(I) \leq \varepsilon \textit{inf}_I\abs{f'(x)},
 \end{align*}
 
 where $J=[ \inf_{I} f, \; \sup_{I} f]$, then $f$ is $(\varepsilon, k)$-uniformly stretching on $I$.
\end{lemma}

Let us define the set, for $q\geq 16$ and $0\leq c<1$,
 \begin{equation}
\label{eq_Mn}
 M_{q,c}= \bigcup_{k=1}^{4q^2}{\left[\frac{k}{4q^2}- \frac{1}{4q^{5/2}}+ \frac{c}{2}, \frac{k}{4q^2}+\frac{1}{4q^{5/2}}+ \frac{c}{2}\right]} \subset \T.
 \end{equation}
 \vspace{0.1cm}
The criterion in Lemma \ref{lemma_stretch_modified} allows us to prove the following result.
%
\begin{lemma}
 \label{lemma_stretching}
 Given $\rho, \delta>0$, $\omega=p/q \in \Q$, $l\in \N$ and $U\in \text{Diff}^{\omega}_{\infty}(\mathbb{S}^3)$ there exist
 $l < \tilde m \in \N$, $\tilde \omega=\tilde p/ \tilde q\in \Q$, $A>0$ and a partial $q$-decomposition $\eta_q$ 
  such that for all $I \in \eta_q$ 
  \begin{itemize}
   \item[i)] $\lambda^2(\Phi^{\tilde m}_{q,A,\tilde \omega}(I))=1$,
   \item[ii)] if $I\in \tilde{\eta}_{q,c} \times \{ \xi\} \subset \eta_q$ then  $\Pi_{\theta_1} I \notin M_{q,c}$ ,
   \item[iii)] for any interval ${J}\subset \Phi_{q,A, \tilde \omega}^{\tilde m}(I)$ we have
   \begin{equation}
  \label{eq_stretch_phi}
  \bigg{|} \cfrac{\lambda^1(I \cap \Phi_{q,A, \tilde \omega}^{- \tilde m}({J}))}{\lambda^1(I)}- \lambda^2({J}) \bigg{|} \leq \rho \lambda^2({J}),
 \end{equation}
 \item[iv)] $\abs{U^{-1}\circ \Phi^{i}_{q, A, \tilde{\omega}}\circ U - U^{-1}\circ \varphi^i_{\omega}\circ U}_{\Delta}<\delta, \quad 1<i\leq l$.
 \end{itemize}
\end{lemma} 

\vspace{0.2cm}

\begin{proof}
Let us first write explicitly, for any $m\in \N$, the $m {\text{th}}$ iterate of the map $\Phi_{q,A, \tilde \omega}$ restricted to the first two coordinates of
a connected set $I \subset N_{q,c }\times \{\xi\}$, for any $\xi \in \mathcal{F}_q$, $c\in [0,1)$. It follows from a computation that, for any $m\in \N$
\begin{small}
\begin{align*}
&\Phi_{q, A, \tilde \omega}^{m} \big{|}_{I}
= (\theta_1+ m \tilde \omega, \theta_2 +m \tilde \omega)\\
&+A r_1^{q(q+1)} r_2^{q^2}(\cos(2\pi q^2 (2\theta_1-c)), (1+ q^{-1})\cos(2\pi q^2 (2\theta_1-c)))\\
&-A r_1^{q(q+1)} r_2^{q^2}(\cos(2\pi (q^2(2\theta_1-c)+ qm \tilde \omega)), (1+ q^{-1})\cos(2\pi (q^2(2\theta_1-c)+ qm \tilde \omega))). 
\end{align*}
\end{small}
Let us define (we assume without loss of generality $q$ not to be a factor of $\tilde{q}$)
\begin{equation}
\label{eq_sequence_m}
\tilde m:=\min \left\{ \tilde{q} \leq m\leq 2\tilde q \; s.t. \; \inf_{k\in \Z}\bigg{|}q m \tilde{\omega} -\cfrac{1}{2}+k \bigg{|}\leq \frac{1}{\tilde q}\right\}.
\end{equation}
For this particular iterate we obtain $\Phi_{q, A, \tilde \omega}^{\tilde m} \big{|}_{I}(\theta_1, \theta_2)
= (f_1(\theta_1) \; \text{mod} \; 1, f_2(\theta_1) \; \text{mod} \; 1)$, with
\begin{align}
 &f_1(\theta_1)= \theta_1+ \tilde m\tilde \omega + 2A r_1^{q(q+1)} r_2^{q^2}(\cos(2\pi q^2(2\theta_1-c))-\sigma(\theta_1)/2), \label{eq_f1}\\
 &f_2(\theta_1)= (q^{-1}-1)\theta_1+c +\tilde m\tilde \omega +2(1+ q^{-1})A r_1^{q(q+1)} r_2^{q^2}(\cos(2 \pi q^2(2\theta_1-c)) - \sigma(\theta_1)/2), 
 \label{eq_f2}
\end{align}
where \[\sigma(\theta_1):=\cos(2\pi (q^2(2\theta_1-c)+ q \tilde{m} \tilde \omega))+\cos(2\pi q^2(2\theta_1-c)).\]
By the mean value theorem and assuming
$\tilde q$ to be sufficiently large w.r.t. $A$ and $q$, we can assume without loss of generality that 
\begin{equation}
\label{eq_wrt_qA}
4A r_1^{q(q+1)}r_2^{q^2}\abs{\sigma'}\leq 1, \quad
4A r_1^{q(q+1)} r_2^{q^2}\abs{\sigma''}\leq 1.
\end{equation}

Consider a connected set $I \subset N_{q,c }\times \{\xi\}$ such that
$\Pi_{\theta_1} I \subset \mathbb{T} \setminus M_{q,c}$. Assume also that $\lambda^2(\Phi^{\tilde m}_{q,A, \tilde \omega}(I))=1$. Since $f_2$ is a monotonic function on the connected sets of $\T \setminus M_{q,c}$, the latter assumption can always  be satisfied if it does not imply that the size of $\Pi_{\theta_1} I$ is larger than the size of the connected components in  $\T \setminus M_{q,c}$. That this is not the case can be ensured if  the derivative of $f_2$ is uniformly sufficiently large w.r.t. $q$, something that we will see one can assume by considering $A$ to be sufficiently large.\newline

It follows from \eqref{eq_xi_decomp}  that the term $r_1^{q(q+1)} r_2^{q^2}$ is uniformly bounded away from zero (with a lower bound depending on $q$) for $\xi \in \mathcal{F}_q$ and 
\begin{equation}
\label{eq_sinus}
\abs{\sin(2\pi q^2(2\theta_1-c))}\geq q^{-\frac{1}{2}}
\end{equation}
on $\T \setminus M_{q,c}$.
The latter inequalities together with the assumption \eqref{eq_wrt_qA} lead to the estimates, for $A$ sufficiently large with respect to $q$,
 \begin{align}
  &\inf_{\T \setminus M_{q,c}}{\abs{f_2'(\theta_1)}}\geq A C_1(q) \label{eq_lapremiere}  ,\\
  &\sup_{\T \setminus M_{q,c}}{\abs{f_2''(\theta_1)}}\leq A C_2(q) \label{eq_ladeuxieme}
 \end{align}
 for some positive constants $C_1(q)$ and $C_2(q)$ depending on $q$.
We have already seen that since  ${\Pi_{\theta_1}I\subset \T \setminus M_{q,c}}$, $I$  is connected and the function $f_2$ is monotonic on $\Pi_{\theta_1}I$ , if $\lambda(f_2(\Pi_{\theta_1} I))=1$ is satisfied we necessarily have that
\begin{equation}
\label{eq_size_interval}
\lambda^1(I)\leq 1/\inf_{\T \setminus M_{q,c}}{\abs{f_2'(\theta_1)}}. 
\end{equation}
In particular if $A$ is sufficiently large w.r.t. $q$ then
\begin{equation}
\label{eq_mida_interval}
\lambda^1(I)\leq q^{-3}
\end{equation}
by \eqref{eq_size_interval} and so $I$ satisfies \eqref{eq_second_cond_dec}. 
It follows from equations \eqref{eq_lapremiere}, \eqref{eq_ladeuxieme} that there exists $C(q)>0$ such that if $A \geq C(q)/\rho$ then
\begin{equation}
\label{eq_implica_stretching}
 \cfrac{\sup_{\theta_1 \in \Pi_{\theta_1} I}{\abs{f_2''(\theta_1)}} \lambda^1(I)}{\inf_{\theta_1\in \Pi_{\theta_1}}{\abs{f_2'(\theta_1)}}} \leq 
 C(q)/A \leq \rho.
\end{equation}
Therefore choosing $A\geq C(q)/\rho$ and sufficiently large so that \eqref{eq_mida_interval} is satisfied, according to Lemma \ref{lemma_stretch_modified} $f_2$ is  $(\rho, 1)$-uniformly stretching on the interval $\Pi_{\theta_1}I$. This implies \eqref{eq_stretch_phi}. Notice that the assumptions above require that we assume that $A$ is sufficiently large w.r.t $q$ and $\rho$, and that $\tilde{q}$ is sufficiently large w.r.t. $q, \rho$ and $A$.  \newline

In short, so far we have proved that there exists a choice of $A$, $\tilde{\omega}=\tilde{p}/\tilde{q}$ and $\tilde{m}$ for which for any $0\leq c <1$ and $\xi \in \mathcal{F}_q$ , any connected set $I \subset N_{q,c}\times \{\xi\}$ with $\Pi_{\theta_1}I \subset \T \setminus M_{q,c}$ and $\lambda(f_2(\Pi_{\theta_1}I))=1$ satisfies that $\lambda^{1}(I)<q^{-3}$ and that $f_2$ is $(\rho,1)$-uniformly stretching on $\Pi_{\theta_1}I$. 
\newline

 Let us now justify that we can define a partial $q$-decomposition $\eta_q$ with all $I\in \eta_q$ satisfying the conditions  $i), ii)$ and $iii)$. 
It follows from combining the estimate \eqref{eq_mida_interval} and the fact that the set $M_{q,c}$ has a small measure, namely $\lambda(M_{q,c})=2/\sqrt{q}$, that we can indeed choose a collection of connected sets in $\tilde{\eta}_{q,c}\times \{\xi\}$ for every $0\leq c  <1$ and $\xi \in \mathcal{F}_q$ as above in such a way that equation \eqref{eq_first_cond_dec} is satisfied (so that their union is large in measure). Condition \eqref{eq_second_cond_dec} follows directly from \eqref{eq_mida_interval}. Thus according to Definition \ref{def_partial_q_decomposition} such a collection of sets is a partial $q$-decomposition. \newline

Assume that we have restricted the choice of parameters $A, \tilde{\omega}$ and $\tilde{m}$ in such a way that they satisfy the relations above, and notice that we can assume without loss of generality that $\tilde{q}$ is arbitrarily large.
We can thus assume that $\tilde{\omega}= \tilde{p}/\tilde{q}$ was chosen with $\tilde{q}$ sufficiently large w.r.t. $l$, $q$ and $A$ in such a way that 
$\norm{\text{D} g_{q,A}\circ U}_{B_{R}} \; l \;  \abs{\tilde{\omega}-\omega}<\delta$, where $R>0$ is sufficiently large so that $g_{q,A}\circ U(B_{\Delta})\subset B_R$. Then we have
\begin{align*}
&\abs{U^{-1}\circ g_{q,A}^{-1} \circ \varphi^i_{\tilde{\omega}}\circ g_{q,A} \circ U - U^{-1}\circ \varphi^i_{\omega}\circ U}_{\Delta} \\
&=\abs{U^{-1} \circ g_{q,A}^{-1} \circ \varphi^{i}_{\tilde{\omega}}\circ g_{q,A} \circ U - U^{-1}\circ g_{q,A}^{-1} \circ \varphi^i_{\omega}\circ g_{q,A} \circ U}_{\Delta}
<\norm{\text{D} g_{q,A}\circ U}_{B_{R}} \; l \;  \abs{\tilde{\omega}-\omega}<\delta
\end{align*}
for $1<i\leq l$. This implies $iv)$.
At the same time, requiring $\tilde{q}$ to be even  larger if needed also allows us to fulfil the requirement $l<\tilde{m}$, because it follows from  \eqref{eq_sequence_m} that $\tilde{m}\to \infty$ as $\tilde{q}\to \infty$. This finishes the proof.
\end{proof}

\section{Proof of the inductive proposition}
\label{sec_prop}
In this section we complete the proof of Theorem \ref{thm_main}
by proving Proposition \ref{prop_dunaalaltra}. We will need a couple of preliminary lemmas.
Let us show first that Proposition \ref{prop_unique_erg} implies that for any $\varepsilon>0, \alpha= p/q\in \T$ and $N\in \N$ sufficiently large, we can find a $(q, \varepsilon, N)$-ergodic map arbitrarily close to $\varphi_{\alpha}$ in the analytic topology. This is a particular case of the following statement.  

\begin{lemma}
 \label{lemma_ergo}
 For any $\mathbbm{a}={p}/{q} \in \Q$, $\mathbbm{e}, \mathbbmtt{k}>0$, $\mathbbm{U}\in \text{Diff}^{\omega}_{\infty}(\mathbb{S}^3, \mu)$ and $\mathbbm{m}\in \N$ there exist
 $\mathbbm{a}'=p'/q'\in \Q$ and $\mathbbm{H}\in \text{Diff}^{\omega}_{\infty}(\mathbb{S}^3, \mu)$ such that 
  \begin{equation}
  f:= \mathbbm{U}^{-1}\circ \mathbbm{H}^{-1} \circ \varphi_{\mathbbm{a}'}\circ \mathbbm{H} \circ \mathbbm{U}
  \end{equation}
  is $(q, \mathbbm{e} ,q')$-ergodic and $\abs{f^i-\mathbbm{U}^{-1}\circ \varphi^i_{\mathbbm{a}}\circ \mathbbm{U}}_{\Delta}<\mathbbmtt{k}, \; 0 < i \leq \mathbbm{m}$. Furthermore, $q'$ can be assumed to be arbitrarily large independently of $\mathbbm{a}$, $\mathbbm{e}, \mathbbmtt{k}, \mathbbm{U}$, $\mathbbm{m}$ and $\mathbbm{H}$.
\end{lemma}

\begin{proof}
 It follows from Proposition \ref{prop_unique_erg} that we can find a sequence of maps 
 $${H_j= h_j \circ \ldots \circ h_0}\subset \text{Diff}^{\omega}_{\infty}(\mathbb{S}^3, \mu),$$
 $h_{-1}=id$ and a sequence $\{\beta_j\}=\{p_j/q_j \}\subset \Q $ with
 $\beta_0 = \mathbbm{a}$ such that 
 $${\lim_{j \to \infty} H_{j-1}^{-1} \circ \varphi_{\beta_{j}} \circ H_{j-1} }$$ 
 is uniquely ergodic. Thus $$\tilde{f}:=\lim_{j\to \infty} f_j =\lim_{j \to \infty} \mathbbm{U}^{-1}\circ H_{j-1}^{-1} \circ \varphi_{\beta_{j}} \circ H_{j-1} \circ \mathbbm{U}$$ is also 
 uniquely ergodic with respect to $\mu$, because $\mathbbm{U}$ is assumed to be volume preserving.
 For any fixed $j\in \N$, consider a ball $B_{R_j}$ such that $H_j \circ \mathbbm{U} (B_{\Delta}) \subset B_{R_j}$. We can assume without loss of generality that the sequence $\{\beta_j\}_{j=0}^{\infty}$ satisfies
 {\small
 \begin{equation}
  \label{eq_growth}
   \abs{\beta_{j+1}-\beta_{j}}< \frac{1}{2^{j+1}} \min \left\{(\mathbbm{m} \norm{\text{D}(H_j \circ \mathbbm{U})}_{B_{R_j}})^{-1}\mathbbmtt{k}, \; (q_j \norm{\text{D}(H_j \circ \mathbbm{U})}_{B_{R_j}})^{-1} \gamma \mathbbm{e} q_j^{-3/2} \right\}. \\
 \end{equation}
 }
 Then from $i)$ in Proposition \ref{prop_unique_erg}
 \begin{align*}
  &\abs{f^i_{j+1}-f^i_{j}}_{\Delta}=\abs{U^{-1}\circ H_{j}^{-1}\circ \varphi^i_{\beta_{j+1}} \circ H_j \circ \mathbbm{U} - \mathbbm{U}^{-1}\circ H_{j}^{-1}\circ \varphi^i_{\beta_{j}} \circ H_{j}\circ \mathbbm{U}}_{\Delta}\\
  &\leq \norm{\text{D}( H_j \circ \mathbbm{U})}_{B_{R_j}} \mathbbm{m} \abs{\beta_{j+1}- \beta_j}, \quad 0 < i \leq \mathbbm{m}
  \end{align*}
  and
  \begin{equation*}
  \abs{f^i_{j+1}-f^i_{j}}_{\Delta}
  \leq \norm{\text{D}(H_j \circ U)}_{B_{R_j}} (q_j-1) \abs{\beta_{j+1}- \beta_j},
   \quad 0 < i\leq q_j-1.
 \end{equation*}
 Therefore due to \eqref{eq_growth} we obtain that for any $l\in \N$
 \begin{align}
 & \abs{f^i_{l}-\mathbbm{U}^{-1}\circ \varphi_{\mathbbm{a}}^i \circ \mathbbm{U}}_{\Delta}<\mathbbmtt{k}, \quad 0< i \leq \mathbbm{m}, \nonumber\\
 & \abs{\tilde{f}^i-f^i_{l}}_{\Delta}< \gamma \mathbbm{e} q_l^{-3/2}, \quad 0< i \leq q_l-1. \label{eq_distance}
\end{align}
For all $l$ sufficiently large,
due to the fact that $\tilde{f}$ is uniquely ergodic (and so we have uniform convergence of Birkhoff sums), for all $x \in \mathbb{S}^3,  B\in \mathcal{C}_q$ and any $\bar{B}\in \{B_{-2\mathbbm{e}}, B_{2\mathbbm{e}}, B \}$ we have
\begin{equation*}
 \bigg{|}\frac{1}{q_l}\sum_{k=0}^{q_l-1}{\mathbbm{1}_{\bar B} \tilde{f}^{k}(x)- \mu(\bar B)}\bigg{|}< \frac{\mathbbm{e}}{2} \mu(\bar B).
\end{equation*}
Therefore for a fixed $l$ large enough so that the inequalities above hold, using \eqref{eq_dis_bs1}, \eqref{eq_dis_bs2} and \eqref{eq_distance} we obtain
\begin{small}
\begin{equation*}
\sum_{k=0}^{q_l-1}\mathbbm{1}_{B_{-2\mathbbm{e}}} \tilde{f}^{k}(x) \leq \sum_{k=0}^{q_l-1}\mathbbm{1}_{B_{-\mathbbm{e}}} f_l^{k}(x)\leq \sum_{k=0}^{q_l-1}\mathbbm{1}_{B} \tilde{f}^{k}(x)\leq \sum_{k=0}^{q_l-1}\mathbbm{1}_{B_{\mathbbm{e}}} f_l^{k}(x)\leq \sum_{k=0}^{q_l-1}\mathbbm{1}_{B_{2\mathbbm{e}}} \tilde{f}^{k}(x).
\end{equation*}
\end{small}

Then the inequality $\abs{\mu(B)- \mu(B_{\pm 2\mathbbm{e}})}<2 \mathbbm{e} \mu(B)$
leads to
\begin{equation*}
 \label{eq_ineq}
  \bigg{|}\frac{1}{q_l}\sum_{k=0}^{q_l-1}{\mathbbm{1}_{B_{\pm \mathbbm{e} }} f_l^{k}(x)- \mu(B_{\pm \mathbbm{e} })}\bigg{|}< 3\mathbbm{e}  \mu(B).
\end{equation*} 
 This implies that $f:=f_l=\mathbbm{U}^{-1}\circ H_{l-1}^{-1}\circ \varphi_{\mathbbm{a}'}\circ H_{l-1}\circ \mathbbm{U}$, with $\mathbbm{a}':= \beta_l$ and $\mathbbm{H}:=H_{l-1}$ satisfy the conclusions of the lemma. Since all the inequalities above hold for all $l$ large enough, we can indeed assume that $q':=q_l$ has been chosen arbitrarily large independently of $\mathbbm{a}$, $\mathbbm{e} , \mathbbmtt{k}, \mathbbm{U}, \mathbbm{m}$ and $\mathbbm{H}$.
\end{proof}


Following the scheme of the construction explained in Section \ref{sec_idea}, 
we want to prove that we can go from \textit{approximate ergodicity} in the sense of Definition \ref{def_approx_ergodic} to \textit{approximate mixing} in the sense of Definition \ref{def_approx_mix} by composing with the map defined in \eqref{eq_g}. In order to do that, we need some technical lemmas.

\vspace{0.2cm}

\subsection{Choice of a partition}
\vspace{0.1cm}

 Given $\rho, \delta>0$, $\alpha'=p'/q' \in \Q$, $l\in \N$ and $U\in \text{Diff}^{\omega}_{\infty}(\mathbb{S}^3)$, assume that $\Phi^{m'}_{q',A, \alpha''}$ as in \eqref{eq_phi} and a partial $q'$-decomposition $\eta_{q'}$ have been chosen to satisfy the conclusions of Lemma \ref{lemma_stretching} for some $m'= \tilde{m},  A>0$ and $\alpha''=\tilde{\omega}$, where $\omega= \alpha'$.  Our goal in this section is to define a partition $\mathcal{J}$ of $\Phi^{m'}_{q',A, \alpha''}(I)$ for any given $I\in \eta_{q'}$. The properties that these partitions need to satisfy to be able to conclude the proof of Proposition \ref{prop_dunaalaltra} are given by Lemma \ref{lemma_intervals}, which is the main result of this section.
  Recall that
 $D_{\xi,\tilde c}=\{\theta_2=\theta_1+\tilde c, \theta_1\in \T \} \times \{\xi\}$.
 The purpose of the following lemma is to show that, under these assumptions, the image of any  $I\in \tilde{\eta}_{q',c} \times \{\xi\} \subset \eta_{q'}$ by $\Phi^{m'}_{q',A,\alpha''}$ lies close to a shifted diagonal $D_{\xi, \tilde c}$ for some $0\leq \tilde c <1$.
 \begin{lemma}
 \label{lemma_cosa}
  Under the assumptions above, for every $I\in \tilde{\eta}_{q',c} \times \{\xi\} \subset \eta_{q'}$
 there exists $0\leq \tilde c <1$ and  $z_0\in D_{\xi, \tilde c}\cap\Phi^{m'}_{q',A, \alpha''}(I)$  with
 \begin{equation}
 \label{eq_dist_new}
 \sup_{\theta_1 \in \Pi_{\theta_1} I }\abs{f_2(\theta_1)-f_1(\theta_1)-\tilde c}<10/\sqrt{q'},
 \end{equation}
 where $f_1$ and $f_2$ are as in \eqref{eq_f1} and \eqref{eq_f2}.
 \end{lemma}
 
 \begin{proof}
   The proof follows from considering $z_0\in D_{\xi, \tilde c}\cap\Phi^{m'}_{q',A, \alpha''}(I)$, which must exist for some $0\leq \tilde c <1$ and then estimating
 $\sup_{\T \setminus M_{q',c}}\abs{f'_2(\theta_1)-f'_1(\theta_1)}$ and $\lambda^1(I)$ from above using equations \eqref{eq_f1}, \eqref{eq_f2} and \eqref{eq_sinus}. Then by the mean value theorem one obtains the upper bound for $\sup_{\T \setminus M_{q',c}}\abs{f_2(\theta_1)-f_1(\theta_1)-\tilde c}$ in \eqref{eq_dist_new}.
 \end{proof}
 Consider now, for $I\in \eta_{q'}$ with  $\eta_{q'}$ a partial $q'$-decomposition as above, $z_0\in D_{\xi, \tilde c}\cap \Phi^{m'}_{q', A, \alpha''}(I)$ and $\tilde{c}$ as given by Lemma \ref{lemma_cosa}, the partition $\tilde{\mathcal{J}}:= \{\tilde J_k\}_{k=0}^{q'-1}$ of $D_{\xi, \tilde c}$ given by 
 
 \begin{equation*}\tilde{J}_0=\left\{\theta_2=\theta_1+ \tilde{c}, \;  \theta_1 \in \left[z_0-\frac{1}{2q'}, z_0+\frac{1}{2q'}\right)\right\}, \quad \tilde{J}_k:=\varphi^k_{\alpha'}(J_0), \quad 1\leq k \leq q'-1.
 \end{equation*}

\begin{definition}
We call $\mathcal{J}$ the partition of the set $\Phi_{q', A, \alpha''}^{m'}(I)$ given by
\begin{equation}
 \label{eq_partition}
\mathcal{J}
:=\left\{J_k \subset \Phi^{m'}_{q', A, \alpha''}(I) \; |\; \Pi_{\theta_2}(J_k)=
\Pi_{\theta_2}(\tilde{J}_k), \tilde{J}_k \in \tilde{\mathcal{J}} \right\}.
\end{equation}
\end{definition}

The new partition $\mathcal{J}$ is well defined because $f_2$ is monotone on $\Pi_{\theta_1}I$ and $\lambda(f_2(I))=1$ (these properties hold because $I$ belongs to a partial $q'$-decomposition $\eta_{q'}$ as given by Lemma \ref{lemma_stretching}). The reason behind defining the partition $\mathcal{J}$ in this way is that the bound in \eqref{eq_dist_new} will allow us to compare the behaviour of its elements to points of the orbit $\{\varphi^k_{\alpha'}(\nonumber z_0)\}_{k=0}^{q'-1}$.

\vspace{0.2cm}

Having introduced the partitions $\mathcal{J}$, we can now proceed to state Lemma \ref{lemma_intervals}.
\begin{lemma}
\label{lemma_intervals}
  Consider $m\in \N$ and \[f=V^{-1}\circ H^{-1} \circ \varphi_{\alpha'} \circ H \circ V\] $(q, 2\varepsilon,q')$-ergodic for $\alpha=p/q, \alpha'=p'/q' \in \Q,\varepsilon>0, H$ and $V\in \text{Diff}^{\omega}_{\infty}(\mathbb{S}^3, \mu)$ fixed.
  There exists $Q(q,\varepsilon, H,V)>0$ such that if $q'> Q(q,\varepsilon, H,V)$ then for any $\kappa>0$ there exist
  $A>0$, $m' \in \N$, $\alpha''=p''/q'' \in \Q$ and a partial $q'$-decomposition $\eta_{q'}$
  such that:
  \begin{itemize}
   \item[a)] $\Phi_{q', A, \alpha''}$ and $\eta_{q'}$ satisfy the conclusion of Lemma \ref{lemma_stretching} for $\rho=\varepsilon \mu(B)$, $\delta=\kappa$,  $\omega= \alpha'$,  $l=m$, $U= H\circ V$,  $\tilde m=m'$ and $\tilde{\omega}=\alpha''$.
  \item[b)] For any $B\in \mathcal{C}_q$ and $I\in \eta_{q'}$
  the subcollections
  \begin{align*}
  &\mathcal{J}_{B_{\pm \varepsilon}}
  :=\{J_k \in \mathcal{J}\; | \; V^{-1}\circ H^{-1} (J_k)\subset B_{\pm \varepsilon}\}, \; \\
  &\mathcal{J}_{B_{\pm \varepsilon}^c}
  :=\{J_k \in \mathcal{J}\; | \; V^{-1}\circ H^{-1} (J_k)\subset B_{\pm \varepsilon}^c\},
 \end{align*}
with $\mathcal{J}$ as in \eqref{eq_partition} satisfy
 \begin{itemize}
  \item[i)] $\big{|}\sum_{J_k\in \mathcal{J}_{B_{\pm \varepsilon}}}\lambda^2(J_{k})-\mu(B_{\pm \varepsilon}) \big{|}<9\varepsilon \mu(B)$,
  \item[ii)] $\lambda^2(\Phi^{m'}_{q', A, \alpha''}(I)\setminus(\mathcal{J}_{B_{\pm \varepsilon}} \cup \mathcal{J}_{B_{\pm \varepsilon}^c}))< 16 \varepsilon \mu(B)$.
   \end{itemize}
   \end{itemize}
\end{lemma}

\begin{remark}
  \label{remark_referee_not_wrap}
 Notice that in the lemma above it is essential that $\eta_{q'}$ satisfies the conclusions of Lemma \ref{lemma_stretching}, because this is a prerequisite for the partitions $\mathcal{J}$ to be well defined. Thus this is not only an extra feature of Lemma \ref{lemma_intervals}, the very same statement would not be well defined without it. 
\end{remark}

\begin{proofof}{Lemma \ref{lemma_intervals}}
Let us divide the proofs into part $a)$ and part $b)$. \newline

\textbf{\textit{Proof of} $a)$.}
 It follows from Lemma \ref{lemma_stretching} that for $\rho=\varepsilon \mu(B), \delta = \kappa$, $\omega=\alpha'=p'/q'$, $l=m$ and $U=H\circ V$, we can obtain  $m'=\tilde m$, $\alpha''=\tilde \omega$, $A>0$ and a partial $q'$-decomposition $\eta_{q'}$  such that $\Phi_{q', A, \alpha''}$ and $\eta_{q'}$ satisfy the conclusions of Lemma \ref{lemma_stretching}. This proves $a)$.\newline
 
 \textbf{\textit{Proof of} $b)$.} It is left to verify that with our choice of parameters when applying Lemma \ref{lemma_stretching} as above, $b)$ is also satisfied. For any $I\in\eta_{q'}$, $B\in \mathcal{C}_q$, let us show that the collections of intervals $\mathcal{J}_{B_{\pm \varepsilon}}$, $\mathcal{J}_{B_{\pm \varepsilon}^c}$ satisfy $i)$ and $ii)$ if $q'$ is assumed to be sufficiently large with respect to an appropriately chosen constant $Q(q,\varepsilon, H, V)$ depending  on $q, \varepsilon, H$ and $V$. \newline
 
 It follows from the uniform continuity of $V^{-1}\circ H^{-1}$ and \eqref{eq_dist_new} that for $q'$ big enough with respect to $\varepsilon^{-1}$,
 if $V^{-1}\circ H^{-1}(\varphi^k_{\alpha'}(z_0))\in B_{-2\varepsilon}$ then we have that $J_k \in \mathcal{J}_{B_{-\varepsilon}}$, and if
  $V^{-1}\circ H^{-1}(\varphi^k_{\alpha'}(z_0))\in B_{2\varepsilon}^c $ then $J_k \in \mathcal{J}_{B_{\varepsilon}^c}$.
  Let us fix $q'$ bigger than a constant $Q(q,\varepsilon, H,V)$ such that the latter holds.
 This leads to the inequalities
 \begin{equation}
 \label{eq_birkhoff}
  \frac{1}{ q'}\sum_{k=0}^{q'-1}{\mathbbm{1}_{B_{-2\varepsilon}}(V^{-1}\circ H^{-1} (\varphi_{\alpha'}^k(z_0)))}
  \leq \frac{1}{q'} \# \mathcal{J}_{B_{\pm \varepsilon}} \leq \frac{1}{q'}\sum_{k=0}^{q'-1}{\mathbbm{1}_{B_{2\varepsilon}}(V^{-1} \circ H^{-1} (\varphi_{\alpha'}^k(z_0)))}.
 \end{equation}
 Notice also that \[\sum_{J_k\in \mathcal{J}_{B_{\pm \varepsilon}}}\lambda^2(J_{k})=\frac{1}{q'} \# \mathcal{J}_{B_{\pm \varepsilon}}.\]
Therefore using that $f$ is $(q, 2\varepsilon, q')$-ergodic and Lemma \ref{lemma_distancies}, we obtain
 \begin{align*}
 &\big{|}\sum_{J_k\in \mathcal{J}_{B_{\pm \varepsilon}}}\lambda^2(J_{k})-\mu(B_{\pm \varepsilon}) \big{|}  
 \leq \max_{\pm} \bigg{|}
 \frac{1}{q'}\sum_{k=0}^{q'-1}{\mathbbm{1}_{B_{\pm 2\varepsilon}}V^{-1}\circ H^{-1}(\varphi_{\alpha'}^k(z_0))- \mu(B_{\pm 2\varepsilon})} \bigg{|} \\
 &+\abs{\mu(B_{\mp \varepsilon})-\mu(B_{\pm 2 \varepsilon})} 
 \leq 9\varepsilon \mu(B),
 \end{align*}
  which proves $i)$.
  The measure of $\Phi^{m'}_{q', A, \alpha''}(I)\setminus(\mathcal{J}_{B_{\pm \varepsilon}} \cup \mathcal{J}_{B_{\pm \varepsilon}^c})$ satisfies 
 \begin{align*} 
 &\lambda^2(\Phi^{m'}_{q',A, \alpha''}(I)\setminus(\mathcal{J}_{B_{\pm \varepsilon}} \cup \mathcal{J}_{B_{\pm \varepsilon}^c}))\\
 &\leq \frac{1}{q'}\sum_{k=0}^{q'-1}{\mathbbm{1}_{B_{2\varepsilon}}(V^{-1}\circ H^{-1}( \varphi_{\alpha'}^k(z_0)))
 - \mathbbm{1}_{B_{-2\varepsilon}}(V^{-1} \circ H^{-1} (\varphi_{\alpha'}^k(z_0)))}\\& \leq
 \mu(B_{2\varepsilon} \Delta B_{-2\varepsilon})+ 6\varepsilon \mu(B)+ 6\varepsilon \mu(B)
 \leq 16 \varepsilon \mu(B),
 \end{align*}
 where we have used again Lemma \ref{lemma_distancies} and the fact that $f$ is $(q,2\varepsilon, q')$-ergodic. This proves $ii)$ and thus finishes the proof.
\end{proofof}

\subsection{Conclusion of the proof for Proposition \ref{prop_dunaalaltra}} 
We are now finally ready to prove Proposition \ref{prop_dunaalaltra}.

\begin{proofof}{Proposition \ref{prop_dunaalaltra}}
We apply Lemma \ref{lemma_ergo} for $\mathbbm{a}=\alpha=p/q$, $\mathbbm{e}=2\varepsilon, \mathbbmtt{k}=\varepsilon/2 , \mathbbm{U}=V$ and $\mathbbm{m}=m$
to obtain $\alpha'=\mathbbm{a}'=p'/q'\in \Q$ and $H=\mathbbm{H} \in \text{Diff}^{\omega}_{\infty}(\mathbb{S}^3, \mu)$ such that \[f=V^{-1}\circ H^{-1} \circ \varphi_{\alpha'} \circ H \circ V\] is $(q, 2\varepsilon, q')$-ergodic and \begin{equation}
\abs{f^{j}-V^{-1} \circ \varphi^j_\alpha \circ V}_{\Delta}<\varepsilon/2, \; 0 < j\leq m.
           \label{eq_eieiei}                                                                                                                                                                                                                                    \end{equation}
It is clear that now we are in a position in which we can apply Lemma \ref{lemma_intervals} to $m$ and $f$.
Also by Lemma \ref{lemma_ergo} we can assume, for $q, \varepsilon, H$ and $V$ as above and any $L>0$, that $q'>q+L$ and $q'>Q(q,\varepsilon, H,V)$ for $Q(q,\varepsilon, H,V)>0$ as in the statement of Lemma \ref{lemma_intervals}. Therefore when we apply Lemma \ref{lemma_intervals} to $m$ and $f$ as above, we obtain that for
\begin{equation*}
 \label{eq_kapa_3}
 \kappa=\varepsilon/2
\end{equation*}
there exist $A>0$, $\tilde m:=m'\in \N$, $\tilde \alpha :=\alpha''=p''/q''$ ($q''>q'$) and a partial $q'$-decomposition $\eta_{q'}$ such that for every $B\in \mathcal{C}_q$ and $I\in \eta_{q'}$  
we have collections of intervals $ \mathcal{J}_{B_{\pm \varepsilon}}$, $\mathcal{J}_{B_{\pm \varepsilon}^c}$  satisfying the conclusions of Lemma \ref{lemma_intervals}, which in particular imply that $\Phi_{q', A, \tilde{\alpha}}$ satisfies the conclusions of Lemma \ref{lemma_stretching} for $\omega= \alpha'$, $\tilde{\omega}=\tilde{\alpha}$, $\rho=\varepsilon \mu(B)$, $\tilde{m}$, $\delta=\kappa$, $l=m$ and $U=H \circ V$. Recall that
\[F=V^{-1}\circ H^{-1}\circ \Phi_{q',A,\tilde{\alpha}}\circ H \circ V. \]
So far we have clarified our choice of parameters and shown that \eqref{eq_fer_dist} is satisfied. We need to show that $i)$ and $ii)$ in the statement of Proposition \ref{prop_dunaalaltra} hold for our choice of parameters.\newline 

\textbf{\textit{Proof of} } $i)$. Let us define $\tilde{H}:=g_{q',A} \circ H$.
Using that $g^{-1}_{q', A}\circ \varphi^j_{\alpha'} \circ g_{q', A}= \varphi^j_{\alpha'}$, \eqref{eq_eieiei} and $iv)$ in Lemma \ref{lemma_stretching}, we obtain
\begin{align*}
 \abs{F^{j}- V^{-1}\circ \varphi_{\alpha}^{j}\circ V}_{\Delta}
 &\leq \abs{F^{j}-
 V^{-1}\circ \tilde{H}^{-1} \circ  \varphi^j_{\alpha'}
 \circ \tilde{H}\circ V}_{\Delta}+
 \abs{f^j- V^{-1}\circ \varphi^j_{\alpha} \circ V}_{\Delta} \\
 &\leq \varepsilon/2 + \varepsilon/2 = \varepsilon,
\end{align*}
for $0< j \leq m$.  \newline 

\textbf{\textit{Proof of}} $ii)$. It remains to show that $F$ is $(q, q', \varepsilon, \tilde m, H \circ V)$-mixing, i.e. that for all $I \in \eta_{q'}$ and $B\in \mathcal{C}_q$
 \begin{equation}
 \label{eq_want_prove}
  \bigg{|} \cfrac{\lambda^1(I \cap \Phi^{-\tilde m}(H \circ V (B_{\pm \varepsilon})))}{\lambda^1(I)}-\mu(B_{\pm \varepsilon})\bigg{|}\leq 30\varepsilon \mu(B),
 \end{equation}
where $\Phi:=\Phi_{q',A,\tilde \alpha}$. Consider the subcollections of intervals $\mathcal{J}_{B_{\pm \varepsilon}}$ and $\mathcal{J}_{B_{\pm \varepsilon}^c}$ given by Lemma \ref{lemma_intervals}. From $i)$ in the same lemma we obtain
{\small
\begin{align}
 &\bigg{|} \cfrac{\lambda^1(I \cap \Phi^{-\tilde m}(H\circ V(B_{\pm \varepsilon})))}{\lambda^1(I)}-\mu(B_{\pm \varepsilon})\bigg{|} \nonumber \\ & \leq
 {\bigg{|} \cfrac{\lambda^{1}(I \cap \Phi^{-\tilde m}(H\circ V(B_{\pm \varepsilon})))}{\lambda^1(I)}-\sum_{J\in \mathcal{J}_{B_{\pm \varepsilon}}}\lambda^2(J)\bigg{|}} \nonumber 
 +{\bigg{|} \sum_{J\in \mathcal{J}_{B_{\pm \varepsilon}} }\lambda^2(J)-\mu(B_{\pm \varepsilon}) \bigg{|}} \\ &
\leq  {\bigg{|} \cfrac{\lambda^1(I \cap \Phi^{-\tilde m}(H\circ V(B_{\pm \varepsilon})))}{\lambda^1(I)}-\sum_{J\in \mathcal{J}_{B_{\pm \varepsilon}}}\lambda^2(J)\bigg{|}}+  9\varepsilon \mu(B). \label{eq_final_quasi}
\end{align}
}
It follows from the definition of the sets $\mathcal{J}_{B_{\pm \varepsilon}}$ and $\mathcal{J}_{B_{\pm \varepsilon}^c}$  in Lemma \ref{lemma_intervals} that 

\begin{equation*}
 \bigcup_{J \in \mathcal{J}_{B_{\pm \varepsilon}}} I \cap \Phi^{-\tilde m}(J) \subseteq I \cap \Phi^{-\tilde m}(H\circ V(B_{\pm \varepsilon})) \subseteq \bigcup_{J \in \mathcal{J} \setminus \mathcal{J}_{B_{\pm \varepsilon}^c}} I \cap \Phi^{-\tilde m}(J),
\end{equation*}

and hence we can conclude from $ii)$ in Lemma \ref{lemma_intervals} and 
\eqref{eq_stretch_phi} in Lemma \ref{lemma_stretching} that
{\small
\begin{align*}
 &{\bigg{|} \cfrac{\lambda^1(I \cap \Phi^{-\tilde m}(H \circ V (B_{\pm \varepsilon})))}{\lambda^1(I)}-\sum_{J\in \mathcal{J}_{B_{\pm \varepsilon}} }\lambda^2(J)\bigg{|}} \\
 & \leq \max \left\{\bigg{|}\frac{1}{\lambda^1(I)}\sum_{J\in \mathcal{J}_{B_{\pm \varepsilon}}}
 \lambda^1(I \cap \Phi^{-\tilde m}(J))-
   \sum_{J\in \mathcal{J}_{B_{\pm \varepsilon}}}\lambda^2(J)\bigg{|}, \right.\\
   & \hspace{3.5cm} \left. \bigg{|} \frac{1}{\lambda^1(I)} \sum_{J\in \mathcal{J} \setminus \mathcal{J}_{B_{\pm \varepsilon}^c}} \lambda^1(I \cap \Phi^{-\tilde m}(J))-
   \sum_{J\in \mathcal{J}_{B_{\pm \varepsilon}}}\lambda^{2}(J)\bigg{|} \right\}\\
   &\leq \max \left\{\bigg{|}  \frac{1}{\lambda^1(I)} \sum_{J\in \mathcal{J}_{B_{\pm \varepsilon}}}\lambda^1(I \cap \Phi^{-\tilde m}(J))-
   \sum_{J\in \mathcal{J}_{B_{\pm \varepsilon}}}\lambda^2(J)\bigg{|}, \right. \\ 
   & \hspace{3cm} \left. \bigg{|} \frac{1}{\lambda^1(I)} \sum_{J\in \mathcal{J} \setminus \mathcal{J}_{B_{\pm \varepsilon}^c}} \lambda^1(I \cap \Phi^{-\tilde m}(J)))-
   \sum_{J\in  \mathcal{J} \setminus \mathcal{J}_{B_{\pm \varepsilon}^c}}\lambda^2(J)\bigg{|} \right\}\\
   \\
   &+ \lambda^2(\Phi^{\tilde m} (I) \setminus (\mathcal{J}_{B_{\pm \varepsilon}^c}\cup \mathcal{J}_{B_{\pm \varepsilon}}))
   \leq \sum_{J\in \mathcal{J} \setminus \mathcal{J}_{B_{\pm \varepsilon}^c}} \bigg{|}\frac{\lambda^1(I \cap \Phi^{-\tilde m}(J))}{\lambda^1(I)}-\lambda^2(J)\bigg{|}\\&
     +16\varepsilon \mu(B)
    \leq  \#(\mathcal{J} \setminus \mathcal{J}_{B_{\pm \varepsilon}^c}) \varepsilon \mu(B) \lambda^2(J) + 16\varepsilon \mu(B) \leq \varepsilon \mu(B) 
   +16\varepsilon \mu(B) \\& \leq  17\varepsilon \mu(B).
\end{align*}
}
Now using \eqref{eq_final_quasi} we finally obtain
\begin{equation*}
 \bigg{|} \cfrac{\lambda^1(I \cap \Phi^{-\tilde m}( H\circ V(B_{\pm \varepsilon})))}{\lambda^1(I)}-\mu(B_{\pm \varepsilon})\bigg{|}\leq 30\varepsilon \mu(B)
\end{equation*} 
and this finishes the proof.
\end{proofof}

\section{Scheme of the proof for higher dimensional cases}
\label{sec_high}
Let us finish by explaining how to modify the setting of the proof to cover the higher dimensional cases. First of all, 
one can generalize the Hopf coordinates to parametrize the odd dimensional sphere $\mathbb{S}^{2n-1}\subset \C^{n}$. In complex coordinates the points $(z_1, \ldots z_n) \in \mathbb{S}^{2n-1}$, with $z_i=x_i+i y_i=r_ie^{2\pi i \theta_i}$ for $i=1, \ldots, n$,  $\theta=(\theta_1, \ldots, \theta_n) \in \T^d$ are such that the moduli  $r= (r_1, \ldots, r_n)$ belong to the subset of the of the $n$ dimensional sphere
${ S_n^+:=\{ r \in \R^n_{+} \; | \; \sum_{i=1}^{n}{r_i^2}=1, \; 0 \leq r_i \leq 1 \} }$. The moduli variables can then be parametrized in spherical coordinates to obtain the parametrization

\begin{align*}
 & z_1= \cos(\xi_1) e^{2\pi i \theta_1}, \\
 & z_2= \sin(\xi_1) \cos(\xi_2) e^{2\pi i \theta_2}, \\
 &\vdots \\
 & z_{n-1}= \sin(\xi_1) \ldots \sin(\xi_{n-2}) \cos(\xi_{n-1}) e^{2\pi i \theta_{n-1}},   \\
 & z_n= \sin(\xi_1) \ldots \sin(\xi_{n-2})\sin(\xi_{n-1})e^{2\pi i \theta_n} \\
\end{align*}

with $\xi_i \in [0, \pi/2]$ for $i=1, \ldots, n-1$ and $\theta=(\theta_1, \ldots, \theta_n)\in \T^n$.
It follows from a computation that in these new variables the volume $\mu$ can be expressed (we denote again the parametrization by $\psi$ ) as

\begin{equation*}
  \mu(A)= \int_{\psi^{-1}(A)}{f(\xi_1, \ldots, \xi_{n-1}) \; d\theta_1 \ldots d{\theta_n} \; d \xi_1 \ldots d \xi_{n-1}}
\end{equation*}
for some smooth map $f:[0, \pi/2]^{n-1} \to \R$, and so that for sets $A\subset \mathbb{S}^{2n-1}$ with $\psi^{-1}(A)=A_1 \times A_2 \in \mathbb{T}^n \times [0, \pi/2]^{n-1}$ with
$A_1, A_2$ Lebesgue measurable we have that the volume can be expressed as 
\begin{equation*}
 \mu(A)= \bar \lambda(A_1) \times \mu_r(A_2),
 \; \text{where} \; \mu_r(A_2):= \int_{A_2}{f(\xi) d \xi}
\end{equation*}
and $\bar{\lambda}$ denotes the Lebesgue measure on $\T^n$.
The map $\varphi_{\alpha}: \mathbb{S}^{2n-1} \to \mathbb{S}^{2n-1}$ can be analogously defined as
\begin{equation*}
 \varphi_{\alpha}(z_1, \ldots, z_n)= (e^{2\pi i \alpha}z_1, \ldots, e^{2\pi i \alpha}z_n).
\end{equation*}
The maps $g_{q,A}$ can be again defined as
\begin{equation*}
g_{q,A}(z)=\zeta_{q}^{A\chi_{q}(z)}(z), \;  A>0,
\end{equation*}
with $\zeta_q^s(z)=(e^{2\pi i s}z_1, e^{2\pi i (1+q^{-1})s } z_2,  \ldots, e^{2\pi i (1+q^{-1}) s}z_n)$
and $\chi_q$ as in \eqref{eq_nu}. Notice that in this case it is enough that the map $\chi_q$ depends only on $z_1, z_2$ in order to obtain stretching in all the angle variables.  In particular in Lemma \ref{lemma_stretching} the stretching happens in all the angle variables with respect to the first component $\theta_1$. The rest of the proof follows as in the case of $\mathbb{S}^3$ up to minor modifications in the technical lemmas, the expression for the decompositions and their specific constants.

 \vspace{0.2cm}
 \noindent
 \textbf{Acknowledgements:} I am grateful to Bassam Fayad and Maria Saprykina for communicating this problem to me and for their valuable suggestions and advice. Support is acknowledged from the Swedish Research Council (VR 2015-04012).

\bibliographystyle{abbrv.bst}
 \bibliography{referencies.bib}

\end{document}